\documentclass{amsart}
\usepackage{amsmath}
\usepackage{amssymb}
\usepackage{amsaddr}
\usepackage{graphicx}
\usepackage{xcolor}
\usepackage{hyperref}
\usepackage{bm}
\usepackage{enumerate}
\usepackage{cleveref}
\usepackage{dsfont}

\newcommand{\vx}{\bm{x}}
\newcommand{\vy}{\bm{y}}
\newcommand{\vu}{\bm{u}}
\newcommand{\vp}{\bm{p}}
\newcommand{\ve}{\bm{e}}
\newcommand{\vj}{\bm{j}}
\newcommand{\vk}{\bm{k}}
\newcommand{\vr}{\bm{r}}
\newcommand{\vn}{\bm{n}}

\newcommand{\vq}{\bm{q}}
\newcommand{\vv}{\bm{v}}
\newcommand{\vz}{\bm{z}}
\newcommand{\vE}{\bm{E}}

\newcommand{\vA}{\bm{A}}

\newcommand{\vC}{\bm{C}}
\newcommand{\vI}{\bm{I}}
\newcommand{\vJ}{\bm{J}}
\newcommand{\vG}{\bm{G}}
\newcommand{\vH}{\bm{H}}

\newcommand{\vP}{\bm{P}}

\newcommand{\vU}{\bm{U}}

\newcommand{\vxi}{\bm{\xi}}
\newcommand{\vPi}{\bm{\Pi}}
\newcommand{\vPsi}{\bm{\Psi}}

\newcommand{\vpsi}{\bm{\psi}}

\newcommand{\vzero}{\bm{0}}

\newcommand{\bGa}{\bm{\alpha}}

\newcommand{\cA}{\mathcal{A}}
\newcommand{\cG}{\mathcal{G}}
\newcommand{\cF}{\mathcal{F}}
\newcommand{\cI}{\mathcal{I}}
\newcommand{\cO}{\mathcal{O}}
\newcommand{\cB}{\mathcal{B}}
\newcommand{\cW}{\mathcal{W}}
\newcommand{\cR}{\mathcal{R}}

\newcommand{\cE}{\mathcal{E}}
\newcommand{\cvG}{\boldsymbol{\cG}}

\newcommand{\tr}{\top}

\newcommand{\M}[1]{\left({#1}\right)}
\newcommand{\Mb}[1]{\left[{#1}\right]}
\newcommand{\Ma}[1]{\left\langle{#1}\right\rangle}
\newcommand{\Mcb}[1]{\left\{{#1}\right\}}

\renewcommand{\Re}{\mathrm{Re}}
\renewcommand{\Im}{\mathrm{Im}}
\newcommand{\real}{\mathbb{R}}
\newcommand{\complex}{\mathbb{C}}

\renewcommand{\i}{\imath}

\DeclareMathOperator{\Cov}{Cov}
\DeclareMathOperator{\Var}{Var}
\DeclareMathOperator{\sinc}{sinc}
\DeclareMathOperator{\linspan}{span}
\DeclareMathOperator{\cond}{cond}
\DeclareMathOperator{\diag}{diag}

\newtheorem{theorem}{Theorem}
\newtheorem{proposition}{Proposition}
\newtheorem{lemma}{Lemma}
\newtheorem{remark}{Remark}

\begin{document}

\title{Imaging small polarizable scatterers with polarization data}

\author{Patrick Bardsley}
\email{bardsleypt@gmail.com}

\author{Maxence Cassier}
\address{Aix Marseille Univ, CNRS, Centrale Marseille, Institut Fresnel, Marseille, France;
Department of Applied Physics and Applied Mathematics, Columbia University, New York, NY 10027, United States.}
\email{maxence.cassier@fresnel.fr}

\author{Fernando Guevara Vasquez}
\address{Mathematics Department, University of Utah, 155 S 1400 E RM
233, Salt Lake City UT 84112-0090.}
\email{fguevara@math.utah.edu}


\begin{abstract}
We present a method for imaging small scatterers in a homogeneous medium
from polarization measurements of the electric field at an array. The
electric field comes from illuminating the scatterers with a point
source with known location and polarization. We view this problem as a
generalized phase retrieval problem with data being the coherency matrix
or Stokes parameters of the electric field at the array. We introduce a
simple preprocessing of the coherency matrix data that partially
recovers the ideal data where all the components of the electric field
are known for different source dipole moments. We prove that the images
obtained using an electromagnetic version of Kirchhoff migration applied
to the partial data are, for high frequencies, asymptotically identical
to the images obtained from ideal data.  We analyze the image resolution
and show that polarizability tensor components in an appropriate basis
can be recovered from the Kirchhoff images, which are tensor fields. A
time domain interpretation of this imaging problem is provided and
numerical experiments are used to illustrate the theory.  
\end{abstract}
\maketitle

\section{Introduction}
\label{sec:intro}
We consider the problem of imaging a collection of small dielectric
scatterers by illuminating the scene with a point source whose location
is known but that is driven by a random process with known statistical
properties. This is a common assumption in optics, where it is easier to
measure polarization, which is a statistical property of light. The data
we use for imaging is also the polarization measured at an array of
receivers. The polarization state
of a wave measured on the plane $x_1, x_2$ can be described by the
coherency matrix (see e.g. \cite{Mandel:OCQ:1995}), which is the
$2\times 2$ Hermitian matrix of correlations between the
$x_1, x_2$ components of the frequency domain electric field
$\vE(\vx_r,\omega)$:
\begin{equation}
 \vPsi = \langle \vE_\parallel \vE_\parallel^* \rangle,
 \label{eq:cohmat0}
\end{equation}
where $\Ma{\cdot}$ denotes the average over many realizations and
$\vE_\parallel \equiv [E_1,E_2]^T$. An equivalent description of the
polarization is the so called {\em Stokes parameters}:
\begin{equation}
    I = \Ma{|E_1|^2 + |E_2|^2},~
    Q = \Ma{|E_1|^2 - |E_2|^2},~
    U = \Ma{2\Re(E_1\overline{E_2})},~\text{and}~
    V = \Ma{2\Im(E_1\overline{E_2})}.
    \label{eq:stokes}
\end{equation}
Thus the polarization data we assume corresponds to a {\em four
dimensional real field} defined over the array. This field can be
measured directly by a conventional CMOS sensor with successive
experiments involving a combination of linear polarizers and
quarter-wave plates, see e.g.  \cite{Berry:1977:MSP}.

We think of the problem at hand as a vector analogue to imaging with
intensities only.  Indeed, suppose that we had full control of the phase
and amplitude of the dipole moment describing the source and that we
could measure both amplitudes and phases of the $x_1, x_2$ components of
the electric field at the array. Then we could perform experiments with
linearly independent dipole moments and the data at the array would be a
$2\times 2$ complex matrix field, or equivalently an {\em 8 dimensional
real field}. Thus by using polarization data {\em we are foregoing half
of the degrees of freedom} compared to this ideal case.  This is similar
to intensity only imaging, where only a one dimensional quantity (the
magnitude or intensity) is measured for a two dimensional quantity (the
complex representation of a scalar field). The other similarity with
intensity measurements is that the coherency matrix \eqref{eq:cohmat0}
is a quadratic form of the data.

The strategy we use for imaging generalizes the approach for scalar
waves in \cite{Bardsley:2016:KIW} and consists of preprocessing the
polarization data to approximate the ideal data. Since the
preprocessing creates an 8 dimensional real field from a 4 dimensional
one, there are very significant errors. The key is to show that these
errors do not affect images of the scatterers, if we use an
electromagnetic version of Kirchhoff imaging \cite{Cassier:2017:IPD}.
We do this with a stationary phase argument. The images we obtain
are $2 \times 2$ complex matrix fields that contain information about
the polarizability tensors of the scatterers, projected onto an appropriate
basis that is dictated by the relative positions of the source, the
array and the scatterers.

\subsection{Related work}

The resolution analysis of the electromagnetic version of Kirchhoff
migration we present here adapts the analysis in \cite{Cassier:2017:IPD}
to the case of an array receivers with a source away from the array. The
conclusion of \cite{Cassier:2017:IPD} is that well known spatial
resolution estimates that hold for acoustics (see e.g.
\cite{Blei:2013:MSI}) also hold in electromagnetics. The recovery of
polarizability tensor information from Kirchhoff migration images was
done in \cite{Cassier:2017:IPD}, where the data was collected for a
collocated array of sources and receivers.  In our case the sources and
the receivers are at different locations and thus the components of the
polarizability tensor that can be stably recovered differ.

There are other methods for imaging small scatterers in
electromagnetics, including methods using MUSIC \cite{Cheney:2001:LSM}
such as \cite{Ammari:2007:MEI,Ammari::2014:TCE,Borcea:2016:RIE}.  The
effect of noise on the MUSIC images has been analyzed using random matrix
theory in \cite{Borcea:2016:RIE}. Imaging small scatterers is related to
the selective focusing problem considered in \cite{Antoine:2008:FFM}.

The problem of finding the electric field from coherency matrix data is
a generalization of the phase retrieval problem. This is a classic problem in
optics where phases are typically much harder to measure than
intensities. Here we give a non-exhaustive overview of methods for
solving the phase retrieval problem or imaging without phases. We focus
on approaches that are the closest to the method we present here. In
specific situations, uniqueness is guaranteed
\cite{Klibanov:1992:PIS,Klibanov:1994:UPK,Klibanov:2014:UTP}. Methods
for solving this problem include iterative methods that approximatively
recover phases
\cite{Gerchberg:1972:GSA}, redundant expansion in frames
\cite{Casazza:2014:PRV,Balan:2006:OSR} and sparsity promoting algorithms
that exploit the sparsity of an image made of point scatterers, see e.g.
\cite{Chai:2011:AII, Candes:2013:PL, Xin:2015:PLO}. 

A key aspect of our imaging method is that we do not need to retrieve
all phases from the data, since the imaging method we use does not
require it. This is a feature that is also used for intensity only
imaging in
\cite{Novikov:2015:ISI,Moscoso:2016:CIP,Moscoso:2017:MII,Bardsley:2016:KIW,Bardsley:2016:IPC}.
For full aperture data, phase retrieval is done in
\cite{Chen:2016:DIM,Chen:2017:PIR,Chen:2017:DIM} using a similar
preprocessing of the intensity data as we present here. One novelty in
our method is that we image polarizability of scatterers, directly from
polarization data. There are other ways of imaging polarizability in
different physical setups such as
Optical Coherence Tomography from interferometric data
\cite{Elbau:2017:ISP} and Polarimetric Synthetic Aperture Radar (see
e.g. \cite{Freeman:1998:TCS}). Finally we point out that our imaging method
can be seen as correlation based imaging, see e.g.
\cite{Schuster:2009:SI,Garnier:2009:PSI,Garnier:2010:RAI,Garnier:2015:SNR}.

Throughout this paper we assume the scatterers are point-like. This is a
reasonably good approximation as can be seen from small diameter
asymptotic analysis \cite{Vog::AFP:2000,Am::AFP:2001,Ammari::2003:ESD}.
The expansion can actually be carried further in terms of powers of the
diameter and can be used for reconstruction, see e.g.
\cite{Ammari::2013}.

\subsection{Contents}
We start by describing two imaging problems in \cref{sec:problem:setup}:
(a) the ideal ``full data'' case where we can control both the amplitudes
and phases of the source and the receivers and (b) the case where they are
random and we only know their statistical properties. In
\cref{sec:imaging} we introduce the electromagnetic version of Kirchhoff
imaging (that assumes full data is available) and define the preprocessing
step that takes polarization data and gives an approximation of the full
data that can be used for imaging. We prove in \cref{sec:stationary} that in the
high frequency limit, the Kirchhoff image of the preprocessed
polarization data is asymptotically identical to the image obtained
with full data. In \cref{sec:fraunhofer} we do a resolution analysis of
Kirchhoff imaging for electromagnetic waves, that generalizes
the results obtained in \cite{Cassier:2017:IPD} to the case where the sources and
receivers are not collocated. All the analysis up to this point is done
in the frequency domain. The connection with the time domain is done in
\cref{sec:stochillum}, where we show that polarization data can also be
obtained by measuring autocorrelations of the time domain electric field
over long periods of time. Our theoretical results are illustrated with
numerical experiments in \cref{sec:numerics}. We conclude with a
discussion and perspectives for future work in \cref{sec:discussion}.

\section{Problem setup}
\label{sec:problem:setup}
The experiment we consider is depicted in \cref{fig:setup}. We illuminate a
family of point-like 
scatterers located at $\vy_1,\ldots, \vy_n \in \mathbb{R}^3$, 
with the field emanating from an electric dipole at a known location
$\vx_s$. We measure the
electric field at an array $\cA$ located in the $x_3=0$ plane.  In the
following, we denote by  $\varepsilon$ the dielectric permittivity, $\mu$
the magnetic permeability, $c = (\varepsilon \mu)^{-1/2}$ the wave
propagation of the medium (assumed to be homogeneous), $\omega$ the
angular frequency and $k=\omega/c$ the wavenumber. We use the convention
\begin{equation}
\vE(\vx;t) = \int d\omega \vE(\vx;\omega) \exp[-\imath \omega t]
~\text{and}~
\vE(\vx;\omega) = \frac{1}{2\pi} \int dt \vE(\vx;t) \exp[\imath \omega t],
\label{eq:fourier:convention}
\end{equation}
to relate the electric field in the time and frequency domain. We assume the symmetry $\vE(\vx;-\omega) = \overline{\vE(\vx;\omega)}$ so that the time domain electric field is real.
\begin{figure}
 \begin{center}
 \includegraphics[width=0.6\textwidth]{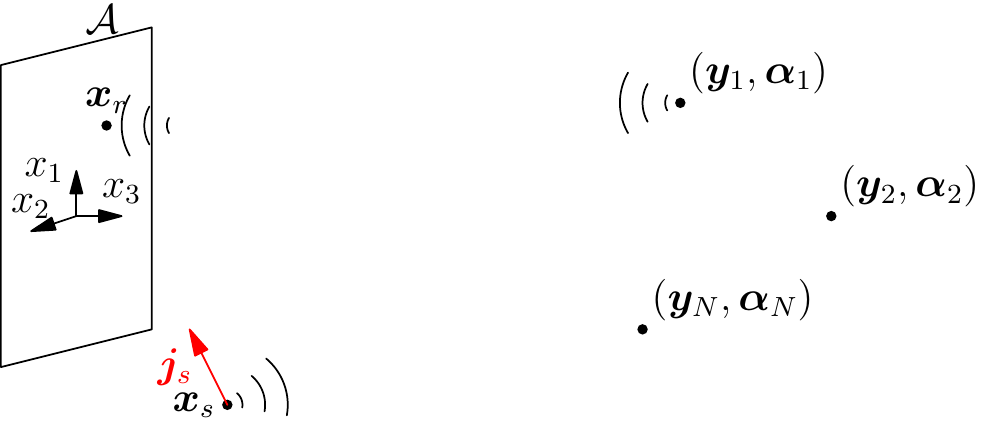}
 \end{center}
 \caption{We seek to image the position and polarizability tensors of a collection of point scatterers at locations
 $\vy_1,\ldots,\vy_n$ from measurements of the electric field made at an
 array $\cA$. The medium is illuminated by an electric dipole at
 a known position $\vx_s$ with polarization $\vj_s$.}
 \label{fig:setup}
\end{figure}

The incident electric field $\vE_{I}(\vx, \omega)$ generated by an
electric dipole located at the source point $\vx_s$ with (possibly
complex) {\em dipole moment} or {\em polarization vector} $\vj_s'(\omega) \in \complex^3$ is 
\begin{equation}
 \vE_{I}(\vx;\omega) = \vG(\vx,\vx_s;\omega/c) \vj_s(\omega),
 \label{eq:incident}
\end{equation}
where we use for convenience $\vj_s(\omega) \equiv \mu \omega^2
\vj_s'(\omega)$. Hereinafter, when we refer to 
polarization vector of an electric dipole, we refer to the rescaled
$\vj_s(\omega)$ instead of the physical $\vj_s'(\omega)$.
Here $\vG(\vx,\vy;k)$ is the dyadic Green function of the
homogeneous background medium, a $3 \times 3$ complex symmetric matrix
given by (see e.g.  \cite{Novotny:2012:PNO})
\begin{equation}
 \vG(\vx,\vy;k) = G(\vx,\vy;k) \Mb{  (1+m(kr))\vI - (1+3m(kr))
 \frac{\vr\vr^\tr}{r^2} }
 \label{eq:dyadic}
\end{equation}
where $G(\vx,\vy;k) = \exp[\imath k r]/(4\pi r)$ is the Green function
for the scalar Helmholtz equation in 3D, $\vr = \vx - \vy$, $r =
|\vx-\vy|$ and $m(kr) = (\imath kr - 1) / (kr)^2$.

In this setting, the expression of the total field (incident field  plus
scattered field) at a receiver is given (see \cite{Novotny:2012:PNO})
by 
\begin{equation}\label{eq.exacttotalfieldpoitnscatterers}
\vE(\vx_r;\omega)= \vE_{I}(\vx_r; \omega)+\sum_{n=1}^N
\vG(\vx_r,\vy_n;\omega/c) \, \bGa(\vy_n;\omega)  \, \vE(\vy_n; \omega),
\, \forall \vx_r \in \cA, 
\end{equation} 
where $\bGa(\vy_n;\omega) \equiv \mu \omega^2 \bGa'(\vy_n;\omega)$  is
the {\em polarizability} or {\em polarization tensor} of the scatterer
located at $\vy_n$.  The $3 \times 3$ complex symmetric matrix
$\bGa(\vy_n;\omega)$ fully characterizes the scattering properties of
the $n-$th scatterer.  Its analogue in acoustics is the reflection
coefficient.  We choose to work with the rescaled tensor $\bGa$ since it
better reflects the high frequency behavior of the physical
polarizability $\bGa'$. Indeed for small penetrable inclusions made of a
Drude-Lorentz material, the rescaled $\bGa$ can be regarded as
independent of frequency for high frequencies \cite{Novotny:2012:PNO}.
More generally, this high frequency behavior holds also for a large
class of inclusions whose permittivity  is characterized by a
generalized Lorentz model (i.e., a sum of Lorentz oscillators).
Hereinafter, when we refer to the polarizability tensor, we refer to the
rescaled $\bGa$, and we assume (for simplicity) that it is constant with
respect to $\omega$. 

We assume scattering in the medium is sufficiently weak such that the
first Born approximation holds (see e.g.
\cite{Born:1959:POE,Novotny:2012:PNO}). This amounts to taking
$\vE(\vy_n; \omega)=\vE_I(\vy_n; \omega)$ in the right-hand side of
\cref{eq.exacttotalfieldpoitnscatterers}, along with the assumption that
the leading order correction to the first Born approximation
(corresponding to second order scattering events) satisfies
\[
\Big\|\sum_{\substack{m=1\\ m\neq n}}^{N}
\vG(\vy_n,\vy_m;\omega) \bGa(\vy_m;\omega)  \, \vE_I(\vy_m; \omega)\Big\|
\ll  \|\vE_I(\vy_n; \omega) \|.
\]
Combining this approximation with
\eqref{eq.exacttotalfieldpoitnscatterers} and \eqref{eq:incident}, the
total field at a receiver $\vx_r\in \cA$ becomes
\[
\vE(\vx_r;\omega)= \vE_{I}(\vx_r; \omega)+ \sum_{j=1}^{N}
 \, \vG(\vx_r,\vy_n;\omega/c) \, \bGa(\vy_n)  \,
\vG(\vy_n,\vx_s;\omega/c) \vj_s(\omega).  
\]
Note, this approximation neglects all multiple scattering events. For
strong multiple scattering, one can include higher order terms of the
Born series \cite{Born:1959:POE}, and/or use the Foldy-Lax model (see
e.g, \cite{Cassier:2013:MSA} for acoustics and \cite{Malet:2005:MGM} for
electromagnetism). 

\subsection{The full data problem}
\label{sec:fulldata}
We first consider the ideal case where we measure all components of the total
electric field at the array $\cA$, corresponding to an illumination with an
electric dipole source  at $\vx_s$ with known polarization
vector $\vj_s(\omega)$.  Since different dipole moments $ \vj_s(\omega)$
could be used to control the incident field, one can extract from the
measurements of the total field the array response function 
$\vPi(\vx_r,\vx_s;k) \in \complex^{3\times 3}$ defined by
\begin{equation}\label{eq::data1}
\vPi(\vx_r,\vx_s;k)= \sum_{n=1}^N \vG(\vx_r,\vy_n;k)
\, \bGa(\vy_n)  \,  \vG(\vy_n,\vx_s;k), \  \forall \vx_r \in \cA.
\end{equation}

\subsection{Polarization data problem}
\label{sec:poldata}
When working with light sources, it is hard to control all the
components of the polarization vector $\vj_s$ as we assumed in
\cref{sec:fulldata}. Actually only certain directions of the
polarization vector matter. This is because the field $\vE_I$ emanating
from the electric dipole at $\vx_s$ can be well
approximated\footnote{The plane wave approximation is rigorously
justified in the asymptotic analysis of \cref{sec:fraunhofer}.} in the
vicinity of a point $\vy_0$ far from $\vx_s$ by the plane wave
\begin{equation}
  \vE_I(\vx;\omega) \approx \frac{\vp \times \vk }{|\vr_0|} \exp[\imath \vk \cdot \vx ],
  \label{eq:pwapprox}
\end{equation}
with polarization $\vp = (\vI - \vr_0 \vr_0^T)\vj_s$, wave vector $\vk =
k \vr_0/|\vr_0|$ and $\vr_0 = \vy_0 - \vx_s$. Thus if we are
far away from a source, we may assume that only two orthogonal
components of the polarization vector can be controlled, i.e. that
$\vj_s \in \linspan\Mcb{\vu_1,\vu_2}$, where $\Mcb{\vu_1,\vu_2}$ is a real
orthonormal basis of polarization directions in $(\vy_0 - \vx_s)^\perp$,
and $\vy_0$ is a known and fixed point near the scatterers we wish to
image. Instead of assuming control of the phases and amplitudes of the
vector $\vj_s$, we assume it is a circularly symmetric Gaussian random
vector (see e.g \cite{Goodman:1963:SAB}) satisfying 
\begin{equation}
 \Ma{\vj_s} = \vzero,~
 \Ma{\vj_s \vj_s^*} = \vJ_s = \vU_s \widetilde{\vJ}_s \vU_s^*,
 ~\text{and}~
 \Ma{\vj_s \vj_s^T} = \vzero,
 \label{eq:js:assumption}
\end{equation}
where $\vU_s \equiv [\vu_1,\vu_2]$, $\vJ_s$ is a known $3 \times 3$
Hermitian matrix and $\Ma{\cdot}$ denotes the mean or expectation. The
assumptions on the frequency dependence of $\vj_s$ are made later in
\cref{sec:stochillum}. Note that the $2\times 2$ Hermitian matrix
$\widetilde{\vJ}_s$ in \eqref{eq:js:assumption} is the coherency matrix
encoding the polarization state of the plane wave approximation near
$\vy_0$ of the point source with origin $\vx_s$.

Similarly, if the array where we make measurements is far away from the
scatterers, the scattered field can be approximated by a plane wave with
polarization vector having a range component that is, for all practical
purposes, zero (i.e. $\vp \approx (p_1,p_2,0)$). Thus the scattered
field also has a range component that is very small compared to the
cross-range components\footnote{This is rigorously justified by the
Fraunhofer regime asymptotic study of \cref{sec:fraunhofer}.}. This
motivates measuring  the matrix of correlations between the cross-range
components of the total electric field, i.e. the {\em coherency matrix}
given by
\begin{equation}
 \vPsi(\vx_r;\omega) := \vU_\parallel^* \Ma{ \vE(\vx_r; \omega) \vE(\vx_r;\omega)^* }\vU_\parallel,
 \label{eq:cohmat}
\end{equation}
where $\vU_\parallel := [\ve_1, \ve_2]$.  
Assuming the source satisfies \eqref{eq:js:assumption} and that the Born
approximation holds, the coherency matrix becomes
\begin{equation}
  \begin{aligned}
  \vPsi &= \vU_\parallel^* \Ma{ 
   (\vG + \vPi) 
   \vj_s \vj_s^* 
   (\vG + \vPi)^*
   }\vU_\parallel\\
   &= \vU_\parallel^*
    (\vG + \vPi) 
    \vU_s \widetilde{\vJ}_s \vU_s^*
    (\vG + \vPi)^*
    \vU_\parallel\\
  & = \widetilde{\vG} \widetilde{\vJ}_s \widetilde{\vG}^* +
      \widetilde{\vPi} \widetilde{\vJ}_s \widetilde{\vG}^* +
      \widetilde{\vG}  \widetilde{\vJ}_s \widetilde{\vPi}^* +
      \widetilde{\vPi} \widetilde{\vJ}_s \widetilde{\vPi}^*,
 \end{aligned}
\label{eq:cohmat2}
\end{equation} 
where for clarity we dropped the dependence on the source, receiver and
frequency, and we adopt the notation 
\begin{equation}
   \widetilde{\vG} = \vU_\parallel^* \vG \vU_s
   ~\text{and}~ 
   \widetilde{\vPi} = \vU_\parallel^* \vPi \vU_s.
   \label{eq:GPi:coeff}
\end{equation} 
\section{Imaging method}
\label{sec:imaging}
The imaging method we present here is an electromagnetic adaptation of
Kirchhoff migration (\cref{sec:km}). This method assumes the full data
(as defined in \cref{sec:fulldata}) is available. The strategy to image
with polarization data is explained in
\cref{sec:strategy}.
\subsection{Kirchhoff migration in electromagnetics}
\label{sec:km}
The single frequency electromagnetic version of the Kirchhoff imaging
function that we use here comes from \cite{Cassier:2017:IPD} and
operates on a $3 \times 3$ complex matrix $\vPi$ field defined on the
array. For each point $\vy$ in the imaging window, the imaging function
produces a $3 \times 3$ complex matrix valued field  given by
\begin{equation}
\cI_{KM}[\vPi](\vy;k) = \int_{\cA} d \vx_{r,\parallel}  \,
\overline{\vG(\vx_r,\vy;k)} \, \vPi(\vx_r,\vx_s;k)
\, \overline{\vG(\vx_s,\vy;k)},
\label{eq:defkirch}
\end{equation}
where the integral is done  with respect to the cross-range coordinates
$\vx_{r,\parallel}$, and with a slight abuse of notation $\cA$ denotes
the set representing the planar array $\cA$. When the data $\vPi$ comes
from $N$ point-like scatterers \eqref{eq::data1}, the Kirchhoff image
can be written explicitly as
\begin{equation}
\begin{aligned}
\cI_{KM}&[\vPi](\vy;k)
\\&=\sum_{n=1}^{N} \left[ \int_{\cA} d\vx_{r,\parallel} \,
\overline{\vG\left(\vx_r, \vy;k\right)}
\vG\left(\vx_r,\vy_n;k\right) \right]\, \bGa(\vy_n) \,
\vG\left(\vy_n,\vx_s;k\right) \,
\overline{\vG\left(\vx_s,\vy; k\right)}     \\
&= \sum_{n=1}^{N} \vH_r\left(\vy,\vy_n;k\right) \,
\bGa(\vy_n) \,  \vH_s\left(\vy,\vy_n;k\right)^{\top},
\end{aligned}
\label{eq.defimagingfuction}
\end{equation}
where $\vH_s$ and $\vH_r$  are the $3\times 3$ complex symmetric
matrices defined by: 
\begin{equation}
\begin{aligned}
\vH_s(\vy,\vy';k)&=
\overline{\vG\left(\vx_s,\vy;k\right)} \,
\vG(\vx_s,\vy';k)  
~\text{and}~\\
\vH_r\left(\vy,\vy';k\right)&= \int_{\cA} d\vx_{r,\parallel}  \,
\overline{\vG\left(\vx_r,\vy;k\right)}
\vG\left(\vx_r,\vy';k\right). 
\end{aligned}
\label{eq:H}
\end{equation}

\subsection{Strategy for imaging with polarization data}
\label{sec:strategy}
Our imaging method consists of two steps. The first step is
``preprocessing'' the coherency matrix data \eqref{eq:cohmat2} to
estimate the cross-range components of the full data \eqref{eq::data1}.
The second step is using Kirchhoff migration to image using this
preprocessed data. We prove that the error made in the preprocessing
step does not appear in the Kirchhoff image, and thus the image we get
with coherency matrix data is (asymptotically) identical to the image
we would obtain with full data (see the stationary phase argument in
\cref{sec:stationary}). If $\vPsi$ is the $2\times 2$ complex Hermitian
matrix representing the coherency matrix, the preprocessing is defined
by the mapping $\vp$ that takes a $2 \times 2$ matrix field in the array
and gives the $3 \times 3$ matrix field given by
\begin{equation}
   \vp(\vPsi) := \vU_\parallel [ 
       \vPsi - \widetilde{\vG}\widetilde{\vJ}_s\widetilde{\vG}^*
     ]
     \widetilde{\vG}^{-*} \widetilde{\vJ}_s^{-1}\vU_s^*,
     \label{eq:prepro}
\end{equation}
where we omitted the dependency on source, receiver and frequency of
$\vPsi$ and $\widetilde{\vG}$ (the $2\times 2$ submatrix of $\vG$
defined in \eqref{eq:GPi:coeff}). The preprocessing map is designed to
extract the $2 \times 2$ submatrix $\widetilde{\vPi}$ of $\vPi$ from the
coherency matrix data $\vPsi$. Indeed, applying the preprocessing $\vp$
to the data \eqref{eq:cohmat2} we get
\begin{equation}
  \vp(\vPsi) = \vU_\parallel \widetilde{\vPi} \vU_s^* + \vq,
\label{eq:prepro:result}
\end{equation}
where the error $\vq$ includes antilinear and sesquilinear terms in $\widetilde{\vPi}$:
\begin{equation}
  \vq :=
  \vU_\parallel [
  \widetilde{\vG}  \widetilde{\vJ}_s \widetilde{\vPi}^* +
  \widetilde{\vPi}  \widetilde{\vJ}_s \widetilde{\vPi}^*
   ] \widetilde{\vG}^{-*}  \widetilde{\vJ}_s^{-1}\vU_s^*.
  \label{eq:prepro:error}
\end{equation}
We prove later in \cref{sec.physmeas} that it is possible to image with
$\vU_\parallel \widetilde{\vPi} \vU_s^*$ instead of $\vPi$.
Finally note that the preprocessing requires to calculate the inverse of
the $2\times 2$ matrix $\widetilde{\vG}^*$ for every point of the array.
This matrix is in general invertible, with a condition number depending
only the relative positions of the array, the source and the scatterers,
as is proved in \cref{app:gtilde}.  We also assume invertibility of the coherency
matrix $\widetilde{\vJ}_s$ of the source. This can be
guaranteed by doing two experiments with different source polarizations
(e.g. two orthogonal linear polarizations) or by considering a source
with sufficiently rich polarizations.

\section{Kirchhoff imaging of preprocessed coherency matrix}
\label{sec:stationary}
We now use a stationary phase argument to show that the Kirchhoff image
of the preprocessed coherency matrix \eqref{eq:cohmat} for a collection
of point scatterers is asymptotically close to the Kirchhoff image
\eqref{eq::data1} obtained  in the same experimental setup by assuming
that one can measure the corresponding field components with the data
$\vU_\parallel \widetilde{\vPi} \vU_s^*$.  Thus, the missing phase
information in the data $\vp(\vPsi)$ does not affect the Kirchhoff image
when the frequency is sufficiently high. In the following, the box $\cW
= [-b/2,b/2]^2 \times [L-h/2,L+h/2]$ is the imaging window that is
assumed to contain all the scatterers. For simplicity we also assume
that $\bGa$ and $\vJ_s$ are constant in frequency. This assumption can
be easily relaxed to include smooth frequency dependence provided the
growth of these quantities (and all their derivatives) as $\omega \to
\infty$ is dominated by a polynomial.

\begin{figure}[h!]
  \includegraphics[width=0.8\textwidth]{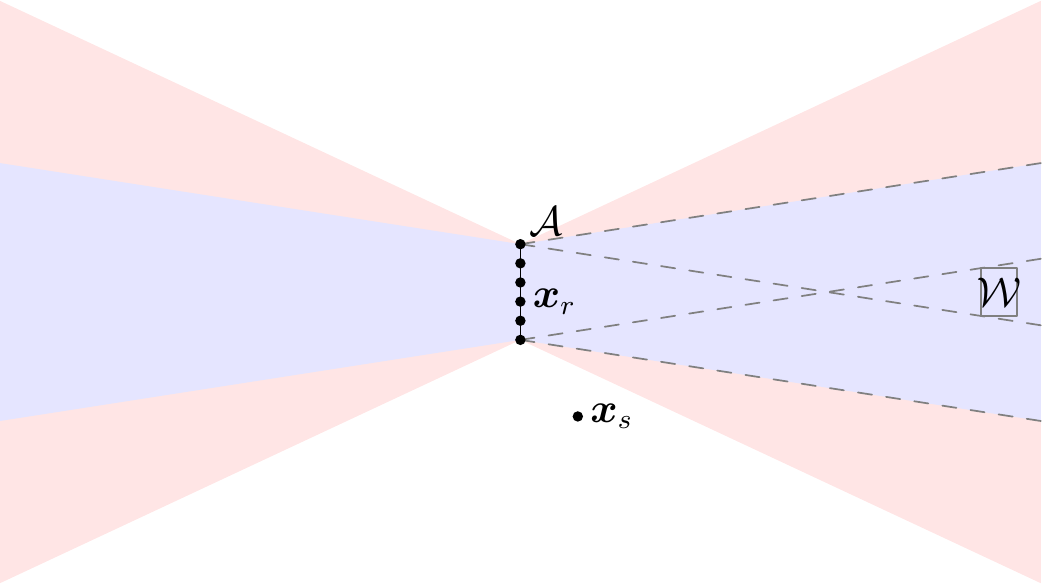}
  \caption{In \cref{thm:geoimg}, the source $\vx_s$ needs to be outside
  of the region $\cR(\gamma)$ given in \eqref{eq:geoimg}. A two dimensional
  example of $\cR(1)$ is given in light blue and is constructed by taking
  the union of identical cones with tips at the array. The cones are
  sufficiently wide so that each one contains the imaging window. The
  positive part of the cones with tip at the boundary of the array are
  in dashed line. The light red areas need to be added to deal with
  multiple scattering, where the region should be $\cR(3)$.}
  \label{fig:geoimg}
\end{figure}
\begin{theorem}
  \label{thm:geoimg}
  For all points $\vy$ in the imaging window $\cW$ we have
  \begin{equation}
   \cI_{KM}[\vp(\vPsi)](\vy;k) = \cI_{KM}[\vU_\parallel \widetilde{\vPi} \vU_s^* ](\vy;k) + o(1),~\text{as}~k \to \infty,
  \end{equation}
  the source location
  $\vx_s$ is {\em outside} of the region $\cR(\gamma)$ defined by
  \begin{equation}
   \cR(\gamma) = \bigcup_{\vx_r \in \cA}\Mcb{ \vy ~|~ 
    |\vx_{r,\parallel} - \vy_\parallel| \leq c \gamma|\vx_{r} - \vy|
   },
    \label{eq:geoimg}
  \end{equation}
  where 
  \begin{equation}
   c = \frac{a+b}{\sqrt{(2L-h)^2 + (a+b)^2}}.
   \label{eq:c}
  \end{equation} 
  If weak scattering holds (i.e. Born approximation) we  can take
  $\gamma = 1$. Alternatively, if finitely many scattering
  events are kept, $\gamma = 3$ is sufficient.  See
  \cref{fig:geoimg} for a visualization of the region $\cR(\gamma)$ for
  different $\gamma$.
\end{theorem}
\begin{proof}
We recall from \eqref{eq:prepro:result} that the preprocessing of
the coherency matrix data $\vPsi$ gives  $\vU_\parallel \widetilde{\vPi}
\vU_s^*$ plus the error $\vq$ given in \eqref{eq:prepro:error}. Hence it
suffices to show that 
\begin{equation}
    \cI_{KM}[\vq;k](\vy) \to 0,~\mbox{as}~k\to \infty.
    \label{eq:km:err}
\end{equation}
To show how to handle multiple scattering, we use the second Born
approximation. Higher order (but finite) Born approximations can be
considered with identical hypothesis. We assume the
scattered field is $\vPi = \vPi_1 + \vPi_2$ where $\vPi_1$ (resp.
$\vPi_2$) consists of single (resp. double) scattering events. From
\eqref{eq.exacttotalfieldpoitnscatterers} $\vPi_2$ is given by 
\begin{equation}\label{eq::data1stcorrection}
 \vPi_2(\vx_r,\vx_s;k)
 = \sum_{n,m,m\neq n} \vG(\vx_r,\vy_n;k) \, \bGa(\vy_n)  \ \vG(\vy_n,\vy_m;k)\, \bGa(\vy_m) \, \vG(\vy_m,\vx_s;k).
\end{equation}
The preprocessing error \eqref{eq:prepro:error} is
the sum of the terms 
\begin{equation}
\vq_j = \vU_\parallel \widetilde{\vG} \widetilde{\vJ}_s
\widetilde{\vPi}_j^* \widetilde{\vG}^{-*} \widetilde{\vJ}_s^{-1} \vU_s^*
~\text{and}~
\vq_{i,j}  = \vU_\parallel \widetilde{\vPi}_i \widetilde{\vJ}_s
\widetilde{\vPi}_j^* \widetilde{\vG}^{-*} \widetilde{\vJ}_s^{-1}
\vU_s^*,
~i,j=1,2.
\end{equation}
{\bf Error term $\vq_1$:} The Kirchhoff image of $\vq_1$ is a sum of oscillatory integrals of the
form 
\begin{equation}
   \int d\vx_{r,\parallel} \vC(\vx_r,\vx_s;k) \exp[\imath k \phi(\vx_r,\vy)],
   \label{eq:statphase1}
\end{equation}
 where the phase is
 \begin{equation}
   \phi(\vx_r,\vy) = 2|\vx_r - \vx_s| - |\vx_r - \vy_*| - |\vx_s - \vy_*| - |\vx_r - \vy| - |\vx_s - \vy|,
   \label{eq:phase1}
 \end{equation}
 where $\vy_* \in \cW$ is one of scatterers
 and $\vC(\vx_r,\vx_s;k)$ is a complex symmetric matrix that is smooth
 for $\vx_r \in \cA$  and is supported in $\cA$ (as a function of
 $\vx_r$). We make this assumption to ignore any boundary terms arising
 from the stationary phase method. The ``amplitude'' matrix
 $\vC(\vx_r,\vx_s;k)$ can be expanded as a power series for $k \neq 0$:
 \begin{equation}
   \vC(\vx_r,\vx_s;k) = \sum_{j=0}^\infty k^{-j}\vC_j(\vx_r,\vx_s),
   \label{eq:amplitude}
 \end{equation}
 uniformly for $\vx_r \in \cA$. The matrix valued terms
 $\vC_j(\vx_r,\vx_s)$ are independent of $k$ and are smooth in $\vx_r$
 because $\vx_r \neq \vx_s$. Their explicit expression does not matter
 for the argument. We can apply the stationary phase method to
 approximate \eqref{eq:statphase1} because $\vC(\vx_r,\vx_s;k)$ and all
 its derivatives with respect to $\vx_{r,\parallel}$ are bounded as
 $k\to \infty$ (see e.g. \cite{deHoop:2018:MSP}). We deduce that
 $\cI_{KM}[\vq_1](\vy) \to 0$ as $k\to \infty$ provided there are no
 points $\vx_r \in \cA$ for which the phase \eqref{eq:phase1} is
 stationary, i.e. there are no points $\vx_r \in \cA$ such that the gradient
 \begin{equation}
   \nabla_{\vx_{r,\parallel}} \phi = 
     2 \frac{\vx_{r,\parallel} - \vx_{s,\parallel}}{|\vx_r - \vx_s|} 
     - \frac{\vx_{r,\parallel} - \vy_{*,\parallel}}{|\vx_r - \vy_*|}
     - \frac{\vx_{r,\parallel} - \vy_{\parallel}}{|\vx_r - \vy|},
     \label{eq:geoimg1}
\end{equation}
vanishes. To guarantee there are no stationary points\footnote{The
stationary argument we use here corrects the similar one used in
\cite{Bardsley:2016:KIW}, which has a region $\cR$ that may be too
small.}, we use the known positions of the array and the imaging window
to define the region $\cR(1)$ in \eqref{eq:geoimg} as a union for $\vx_r
\in \cA$ of the cones $\{\vy ~|~ |\vx_{r,\parallel} - \vy_\parallel |
\leq c |\vx_r - \vy|\}$, where $c$ is chosen so that each cone contains
the imaging window $\cW$ and is given by \eqref{eq:c}, see
\cref{fig:geoimg}. The expression \eqref{eq:c} comes from an elementary
geometric argument involving the cones with tips at the boundary of the
array (in dashed line in \cref{fig:geoimg}). If $\vx_r$ is a stationary point then there are
points $\vy,\vy_* \in \cW$ for which \eqref{eq:geoimg1} is zero.
Isolating the term corresponding to $\vx_s$ and taking norms we get
\begin{equation}
2 \frac{|\vx_{r,\parallel} - \vx_{s,\parallel}|}{|\vx_r - \vx_s|} 
\leq
     \frac{|\vx_{r,\parallel} - \vy_{*,\parallel}|}{|\vx_r - \vy_*|}
   + \frac{|\vx_{r,\parallel} - \vy_{\parallel}|}{|\vx_r - \vy|}
\leq 2 c,
\label{eq:xsincone}
\end{equation}
where the last inequality follows from $\vy,\vy_* \in \cW \subset \cR(1)$.
We conclude that $\vx_s \in \cR(1)$. To summarize, if $\vx_s \notin
\cR(1)$, the Kirchhoff image of $\vq_1$ vanishes as $k \to \infty$
(faster than any polynomial in $k^{-1}$).

{\bf Weak scattering assumption:} If we operate under the weak
scattering assumption, we are done since we can neglect $\vq_2$,
$\vq_{i,j}$ and $\vPi_1 \vPi_1^*$.

{\bf Error term $\vq_2$:} The term $\vq_2$ is composed of oscillatory integrals of the form
\eqref{eq:statphase1} but with phase given by
\begin{equation}
 \phi(\vx_r,\vy) = 2|\vx_r - \vx_s| - |\vx_r - \vy_n| - |\vy_n - \vy_m| -
 |\vy_m - \vx_s| - |\vx_r - \vy| - |\vy - \vx_s|.
 \label{eq:phase2}
\end{equation}
The gradient of \eqref{eq:phase2} with respect to $\vx_{r,\parallel}$ is
given by
\begin{equation}
\nabla_{\vx_{r,\parallel}} \phi  = 
    2\frac{\vx_{r,\parallel} - \vx_{s,\parallel}}{|\vx_r - \vx_s|} 
     - \frac{\vx_{r,\parallel} - \vy_{n,\parallel}}{|\vx_r - \vy_n|}
     - \frac{\vx_{r,\parallel} - \vy_{\parallel}}{|\vx_r - \vy|}.
\end{equation}
If  there are points $\vy_n,\vy
\in \cW \subset \cR(1)$ such that $\nabla_{\vx_{r,\parallel}} \phi  =
0$, we can conclude in a way similar to \eqref{eq:xsincone} that
$\vx_s \in \cR(1)$. Taking $\vx_s \notin \cR(1)$ guarantees
that there are no stationary phase points.

{\bf Error term $\vq_{i,j}$:} This term is composed of oscillatory
integrals of the form \eqref{eq:statphase1} with phase
\begin{equation}
 \phi(\vx_r,\vy) = |\vx_r - \vx_s| + \phi_i -
 \phi_j' - |\vx_r - \vy| - |\vx_s - \vy|,~i,j=1,2
\end{equation}
where $\phi_1 = |\vx_r - \vy_m| + |\vy_m - \vx_s|$, $\phi_2 = |\vx_r
- \vy_{m}| + |\vy_{m} - \vy_{n}| + |\vy_{n} - \vx_s|$, and $m\neq n$ are
scatterer indices. By $\phi_j'$ we mean putting on the indices $n$ and
$m$, to represent taking a possibly different path
among the scatterers. Its gradient
with respect to $\vx_{r,\parallel}$ is
\begin{equation}
 \nabla_{\vx_{r,\parallel}} \phi = 
  \frac{\vx_{r,\parallel} - \vx_{s,\parallel}}{|\vx_r - \vx_s|} 
 + \frac{\vx_{r,\parallel} - \vy_{m,\parallel}}{|\vx_r - \vy_m|}
 - \frac{\vx_{r,\parallel} - \vy_{m',\parallel}}{|\vx_r - \vy_{m'}|}
 - \frac{\vx_{r,\parallel} - \vy_\parallel}{|\vx_r - \vy|}.
 \label{eq:qijgradphase}
\end{equation}
If there are points $\vy_m,\vy_{m'},\vy \in \cW \subset \cR(1)$ such that
$\nabla_{\vx_{r,\parallel}} \phi = 0$ then proceeding as in
\eqref{eq:xsincone} we get
\begin{equation}
  \frac{|\vx_{r,\parallel} - \vx_{s,\parallel}|}{|\vx_r - \vx_s|} 
  \leq
  \frac{|\vx_{r,\parallel} - \vy_{m,\parallel}|}{|\vx_r - \vy_m|}
 + \frac{|\vx_{r,\parallel} - \vy_{m',\parallel}|}{|\vx_r - \vy_{m'}|}
 + \frac{|\vx_{r,\parallel} - \vy_\parallel|}{|\vx_r - \vy|}
 \leq 3c.
\end{equation}
We conclude $\vx_s \in \cR(3)$ and that taking $\vx_s \notin \cR(3)$
guarantees the absence of stationary points.
\end{proof}

\begin{remark}
Let us call $\vq_j$ the antilinear term involving the $j-$th term in the
Born series and $\vq_{i,j}$ the sesquilinear term involving the $i-$th
and $j-$th terms. Clearly $\nabla_{\vx_{r,\parallel}} \vq_j$ is of the
form \eqref{eq:geoimg1} and  $\nabla_{\vx_{r,\parallel}} \vq_{i,j}$ is
of the form \eqref{eq:qijgradphase}. This is because for any path going
from $\vx_r$, going through a chain of scatterers and then to $\vx_s$,
there is only one segment involving $\vx_r$. Hence the proof of
\cref{thm:geoimg} can be modified to include these higher order terms.
\end{remark}

\begin{remark}
The region $\cR(\gamma)$ that we have to avoid for placing the source in
\cref{thm:geoimg} is a worst case scenario
that assumes no knowledge about the position of the scatterers, other
than them being in $\cW$. This exclusion region can be
smaller since the scatterers are point-like.
\end{remark}

\begin{remark}
Notice that in the proof of \cref{thm:geoimg}, all the phase terms of
the error in the Kirchhoff imaging function never vanish. This is not
the case for the phase of the oscillatory integral appearing in the
image $\cI_{KM}(\vU_\parallel \widetilde{\vPi} \vU_s^*)$. This is
essentially the travel time argument that can be used to motivate
Kirchhoff image (see \cite{Blei:2013:MSI}). In the next section we do a
resolution analysis of the Kirchhoff imaging function, i.e.
characterizing the typical size of a focal spot (in both range and
cross-range).
\end{remark}

\section{Fraunhofer asymptotic analysis}
\label{sec:fraunhofer}

We begin the study of the Kirchhoff imaging function in the Fraunhofer diffraction
regime by giving the assumptions we make about the length scales of the
problem for both the array (\cref{sec.asympparam}) and the source
(\cref{sec.srcasymp}). Under these assumptions, we give expressions for
the data and the imaging function (\cref{sec.kmasymp}) that allow us to
get resolution estimates in both cross-range (i.e. plane parallel to the
array, see \cref{sec.crossrangepola}) and range (i.e. direction normal
to the array, see \cref{sec.rangeesti}). Extracting the polarizability
tensor data from the Kirchhoff images gives oscillatory artifacts that
we correct in \cref{sec.phasecorrection}. Finally we explain in
\cref{sec.physmeas} why it is possible to image with the full data
projected on an appropriate basis.

\begin{figure}
 \begin{center} 
  \includegraphics[width=0.8\textwidth]{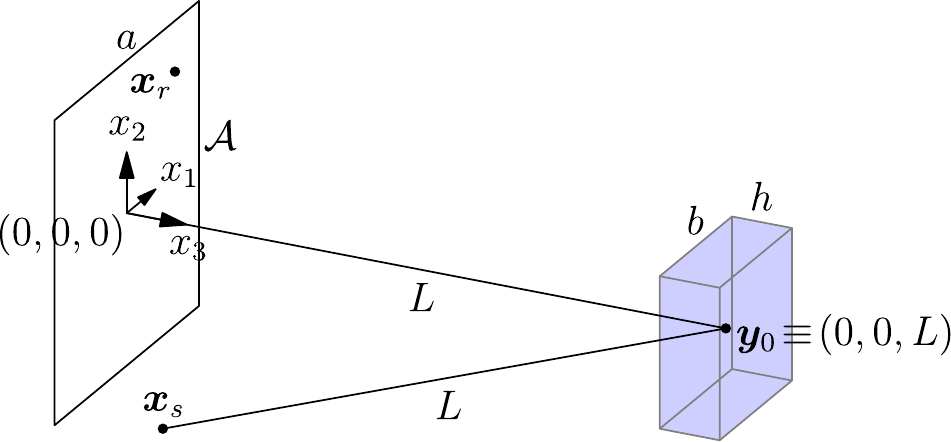}
 \end{center}
 \caption{Fraunhofer regime scalings. The array $\mathcal{A}$ has side
 $a$ and lies in the $x_1,x_2$ plane and the imaging window is a box of dimensions $b$ in cross-range
 and $h$ in range. The point $\vy_0 = (0,0,L)$ is the reference point
 for the imaging window. The source $\vx_s$ is not in the array and
 $|\vx_s - \vy_0| = L$.}
 \label{fig:fraunhofer}
\end{figure}

\subsection{Length scales for Fraunhofer asymptotic regime}
\label{sec.asympparam}
Let $\vy=(\vy_{\parallel}, L+\eta)$ and $\vx_r=(\vx_{r,\parallel},0)$ be
respectively an imaging point and a point in the array, with $L$ being the characteristic
propagation distance. We assume that  the scatterers lie in
a known imaging window  of
characteristic size $b$ in cross-range and $h$ in range, i.e. $|
\vy_{\parallel} |= \cO(b) ~\text{and}~  | \eta |=\cO (h)$, as illustrated
in \cref{fig:fraunhofer}. For the Fraunhofer asymptotic regime we assume
the following scalings hold (see e.g.  \cite{Borcea:2008:EII,Born:1959:POE}).
\begin{itemize}
\item  $kL \gg 1$, (high frequency or large propagation distance)
\item
Fresnel number $\Theta_{a} \equiv \displaystyle \frac{k a^2}{L}\ll k L$,
i.e. small aperture: $a\ll L$,
\item Fresnel number $\Theta_{b} \equiv \displaystyle \frac{k b^2}{L}\ll
kL$, i.e. small imaging window in cross-range: $b\ll L,$
\item Fresnel number $\Theta_{h} \equiv \displaystyle \frac{k h^2}{L}\ll
kL$, i.e.  small imaging window in range $h\ll L$.
\end{itemize}
Moreover we assume that
\[
 \Theta_{b} \ll 1,  ~ %
  ~ 1 \ll \Theta_{a} \ll
 \frac{L^2}{a^2 }, \mbox{ and }   kh=\cO( 1 ),
\]
namely that the imaging window is small compared to
the array aperture (i.e. $b\ll a$) and that the range component $\eta$ of the imaging point is small or of the order of the wavelength $2
\pi/k$.

\subsection{Asymptotic analysis related to the source}
\label{sec.srcasymp}
Let  $\vy_0 \equiv (0,0,L)$ be a reference point in the imaging window.
We assume for simplicity that the source-to-scatterer and
array-to-scatterer distances are identical, i.e. $|\vx_s-\vy_0|=L$.
Nevertheless, all the asymptotic results we present here hold if $L$ and
$|\vx_s-\vy_0|$ have the same order of magnitude. With this assumption,
it is easy to check that the distance between the source $\vx_s$ and an
imaging point $\vy$ satisfies:
\begin{equation}\label{eq.dsitsource}
 |\vy-\vx_s|=L\Big[1+\cO\Big(\frac{b}{L}\Big)+\cO\Big(\frac{h}{L}\Big)\Big] .
\end{equation}

\subsection{Fraunhofer asymptotic analysis of the Kirchhoff imaging function}
\label{sec.kmasymp}
In the Fraunhofer regime (see\cite[eq. (8)]{Cassier:2017:IPD}), the
dyadic Green function between an imaging point $\vy$ and a receiver
$\vx_r$ on the array is given by
\begin{eqnarray}\label{eq.asymptGreenarray}
\vG(\vx_r, \vy;k)
&=& \widetilde{G}(\vx_r, \vy;k)
\Big[\vP(\vx_r,\vy)+\cO\Big(\frac{a^2\Theta_a}{L^2}\Big)+
\cO(\Theta_b)\Big],
\end{eqnarray}
where $\widetilde{G}(\vx_r, \vy;k)$ is the Fraunhofer or paraxial
approximation of the acoustic Green function: 
\[
\widetilde{G}(\vx_r, \vy;k) \equiv \frac{1}{4\pi L}  ~
\exp\Mb{\imath \big( k L+\frac{k\,|\vx_r|^2}{2
L}+\frac{k\vx_{r,\parallel}\cdot\vy_{\parallel}}{L}+k\eta \big)},
\]
and $\vP(\vx,\vy)$ is the orthogonal projector defined by
\[
\vP(\vx,\vy)=\vI-\frac{(\vx-\vy)(\vx-\vy)^{\tr}}{|\vx-\vy|^2},
\]
with $\vI$ being the  $3\times3$ identity matrix. For convenience
we introduce two particular projectors: the projector $\vP_{\parallel}$
on the cross-range direction of $\cA$ and the projector  $\vP_s$ on
$(\vy_0 - \vx_s)^\perp$. We write these projectors in terms of an
orthonormal basis of their range
\begin{equation}\label{eq.projectdef}
\vP_{\parallel}=\vU_{\parallel}\, \vU_{\parallel}^{*} \mbox{ and } \,
\vP_s=\vP(\vx_s,\vy_0)=\vU_s\, \vU_s^{*},
\end{equation}
where $\vU_\parallel\equiv [\ve_1,\ve_2]$ is a $3\times2$ matrix and the columns of the $3\times2$ matrix $\vU_s$
form an orthonormal basis of $(\vy_0 - \vx_s)^\perp$, that we assume is
chosen a priori.

In this regime, the projector $\vP(\vx_r,\vy)$ can be approximated by
\begin{equation}
\vP(\vx_r,\vy)=\vP_{\parallel}+\cO\Big(\frac{a}{L}\Big),
\label{eq:Pxry}
\end{equation}
since $b=o(a)$, $\cO(h/L)=\cO(1/kL)=\cO(a^2/(L^2\Theta_a))=o(a^2/L^2)$
and $|\vx_r-\vy|^{-2}=L^{-2}(1+\cO(a^2/L^2))$.  Thus, relation
\eqref{eq.asymptGreenarray} simplifies to:
\begin{equation}\label{eq.asymptgreenarray}
\vG(\vx_r, \vy;k)
= \widetilde{G}(\vx_r, \vy;k)
\Big[\vP_{\parallel}+\cO\Big(\frac{a^2\Theta_a}{L^2}\Big)+
\cO(\Theta_b)+\cO\Big(\frac{a}{L}\Big) \Big].
\end{equation}

Now, it follows by using \eqref{eq.asymptgreenarray} and integrating on
the array that the matrix $\vH_r\left(\vy,\vy_n;k\right)$ (defined in
\eqref{eq:H}) admits the following asymptotic expansion: 
\begin{equation}\label{eq.spreadfunctionarray}
\vH_r\left(\vy,\vy_n;k\right)=\widetilde{\vH}_r\left(\vy,\vy_n;k\right)
+ \cO\Big(\frac{a^4\Theta_a}{L^4}\Big)+ \cO\Big(\frac{a^2\Theta_b}{L^2}
\Big)+\Big(\frac{a^3}{L^3} \Big),
\end{equation}
where for $\vy=(\vy_{\parallel},L+\eta)$ and $\vy'=(\vy'_{\parallel},L+\eta')$,
 $\widetilde{\vH}_r\left(\vy,\vy';k\right)$ is given by
\begin{eqnarray}\label{eq.defhraproxs}
\widetilde{\vH}_r\left(\vy,\vy';k\right)&=&\frac{\exp[\imath k(\eta'-\eta)]}{(4\pi
L)^2}\int_{\cA} d\vx_{r,\parallel}\, \exp\Mb{\imath k
\left(\frac{\vx_{r,\parallel}\cdot(\vy'_{\parallel}-\vy_{\parallel})}{L}\right)}   \vP_{\parallel}  \nonumber \\
&=&\frac{\exp[\imath k(\eta'-\eta)]}{(4\pi L)^2}  \,
\cF[\mathds{1}_{\cA}]\Big(\frac{k}{L}(\vy'_{\parallel}-\vy_{\parallel})\Big)
\vP_{\parallel}.
\end{eqnarray}
Here $\mathds{1}_{\cA}$ denotes the indicator
function of the array and $\cF[f]$ is the Fourier transform of an
integrable function $f$, defined using the convention 
$$
\cF[f](\vxi_{\parallel})=\int_{\real^2} f(\vx_{\parallel}) \exp[\imath \vx_{\parallel} \cdot \vxi_{\parallel} ]  \,
\, d \vx_{\parallel}, \   \forall  f\in L^1(\mathbb{R}^2).
$$

The dyadic Green function between an imaging point $\vy$ and the source
has asymptotic expansion \cite[eq. (7)]{Cassier:2017:IPD}
\begin{equation}\label{eq.asym1}
\vG(\vx_s, \vy;k)=G(\vx_s, \vy;k)
\Big[\vP(\vx_s,\vy)+\cO\Big(\frac{1}{kL} \Big)\Big].
\end{equation}
One can also show using \eqref{eq.dsitsource} that the projector $\vP(\vx_s,\vy)$ can be approximated in the Fraunhofer regime by
\begin{equation}\label{eq.asym2}
\vP(\vx_s,\vy)=\vP_s+\cO\Big(\frac{b}{L}\Big)+
\cO\Big(\frac{h}{L}\Big).
\end{equation}
Using \eqref{eq.asym1}, \eqref{eq.dsitsource} and \eqref{eq.asym2} we
get that
\begin{eqnarray}\label{eq.aproxsource}
\vG(\vx_s, \vy;k)&=& \frac{ e^{\imath k |\vx_s-\vy|}}{4 \pi L}
\Big[1+\cO\Big(\frac{b}{L}\Big)+\cO\Big(\frac{h}{L}\Big)\Big]
\Big[\vP_s+ \cO\Big(\frac{b}{L}\Big)+\cO\Big(\frac{1}{k\, L}\Big) \Big] \nonumber\\
&=&  \frac{ e^{\imath k |\vx_s-\vy|}}{ 4\pi L}
\Big[\vP_s+O\Big(\frac{b}{L}\Big)+\cO\Big(\frac{a^2}{\Theta_a
L^2}\Big)\Big],
\end{eqnarray}
since $kh=\cO( 1 )$ implies that
$\cO(h/L)=\cO\left(1/(kL)\right)=\cO\left(a^2/(\Theta_a
L^2)\right)$.
Thus, the matrix $\vH_s\left(\vy,\vy_n;k\right)$ admits the asymptotic expansion:
\begin{equation}\label{eq.asympspreadfunctionsource}
\vH_s\left(\vy,\vy_n;k\right)= \frac{1}{(4\pi L)^2} \Big[\widetilde{\vH}_s\left(\vy,\vy_n;k\right)+O\Big(\frac{b}{L}\Big)+\cO\Big(\frac{a^2}{\Theta_a L^2}\Big)\Big],
\end{equation}
where for $\vy=(\vy_{\parallel},L+\eta)$ and $\vy'=(\vy'_{\parallel},L+\eta')$:
\begin{equation}\label{eq.Hsasymp}
\widetilde{\vH}_s\left(\vy,\vy';k\right)= \exp \left[ \imath k (|\vx_s-\vy'|-|\vx_s-\vy|)\right] \vP_s.
\end{equation}
We assume from now on that all polarizability tensors have the same
order of magnitude and that $\|\bGa_n\|=\cO(1)$ for $n=1, \ldots, N$.
The main result of this section is the following asymptotic expansion.
\begin{proposition}\label{prop.asymKI}
The Kirchhoff imaging function applied to full array data from $N$
scatterers admits the asymptotic expansion
\begin{equation*}
\cI_{KM}(\vy;k)=   \frac{1}{(4\pi L)^2} \left[\sum_{n=1}^{N} \widetilde{\vH}_r\left(\vy,\vy_n;k\right) \,
\bGa(\vy_n) \,
\widetilde{\vH}_s\left(\vy,\vy_n;k\right)^{\top}+o(1)\right]
\end{equation*}
in the Fraunhofer regime. The remainder $o(1)$ is given explicitly with respect to the Fraunhofer parameters by:
 $\cO\left(a^4\Theta_a /L^4\right)+\cO( a^2\Theta_b /L^2)+ \cO\left(a^3 /L^3\right)$.
\end{proposition}
\begin{proof}
The proof is an immediate consequence of the definition of the imaging
function \eqref{eq.defimagingfuction} and the asymptotic  formulas \eqref{eq.spreadfunctionarray} and \eqref{eq.asympspreadfunctionsource}.
\end{proof}

\subsection{Cross-range estimates of position and polarizability tensor}\label{sec.crossrangepola}
To study the cross-range resolution of the Kirchhoff image we consider,
without loss of generality, the case of a single dipole located at $\vy_*
= (\vy_{*,\parallel},L + \eta_*)$ with associated polarizability tensor
$\bGa_{*}$.  The following proposition characterizes the decay of the
Kirchhoff image in the cross-range away from the dipole position. 
\begin{proposition}[Imaging
function decrease in cross-range] \label{prop.crossrang}
The Kirchhoff imaging function \eqref{eq.defimagingfuction} of a dipole located at
$\vy_*=(\vy_{*,\parallel}, L+\eta_*)$ and evaluated at  $\vy=(\vy_{\parallel},
L+\eta_{*})$ satisfies
\begin{itemize}
\item  If the imaging point does not coincide with the dipole in cross-range, ie $\vy_{\parallel}\neq \vy_{*,\parallel}$ then
\begin{equation}\label{eq.asymp1}
 \|\cI_{KM}(\vy;k)\| =  \frac{1}{(4\pi L)^2} \left[  \frac{ a^2}{L^2}  \, \left( \cO\Big(  \frac{ L}{
 a \, k   |\vy_{\parallel}-\vy_{*,\parallel} |}\Big) +o(1) \right)\right],
\end{equation}
for any matrix norm and where $o(1)$ is explicitly given by $\cO\left(a^2\Theta_a /L^2\right)+\cO( \Theta_b)+ \cO\left(a/L\right)$.
\item  If the imaging point  coincides with the dipole in cross-range,
i.e. $\vy_{\parallel}= \vy_{*,\parallel}$ then
\begin{eqnarray}\label{eq.asymp2}
\cI_{KM}(\vy_*;k)=\ \frac{1}{(4\pi L)^2} \Big[\frac{  \operatorname{mes} \cA}{(4\pi L)^2} \vP_{\parallel}  \, \bGa_* \, \vP_s+o\Big(\frac{a^2}{L^2}\Big)\Big],
\end{eqnarray}
where  $\operatorname{mes}  \cA = \Theta(a^2)$ is the area of the array
$\cA$.\footnote{The notation $f(x) = \Theta(g(x))$ means that there are
constants $c,C>0$ such that $cg(x) \leq f(x) \leq Cg(x)$ in an
appropriate limit for $x$.}
\end{itemize}
\end{proposition}
Concretely, \cref{prop.crossrang} shows that the image vanishes
asymptotically (compared to its value at $\vy_*$) provided $|\vy_{\parallel}-\vy_{*,\parallel} |$
is large with respect to the Rayleigh number $L/(ak)$. Hence the
characteristic size of the focal spot in the cross-range of a dipole is
given by the Rayleigh resolution $L/(ak)$, as is the case in acoustics
(see e.g. \cite{Blei:2013:MSI,Borcea:2007:OWD}) or in electromagnetics
with collocated sources and receivers \cite{Cassier:2017:IPD}.  The
proof of this proposition is based on the following lemma which
approximates the matrix
$\widetilde{\vH_r}\left(\vy,\vy';k\right)$ in Fraunhofer regime.
\begin{lemma}\label{lem.asymHr}
For $\vy=(\vy_{\parallel},L+\eta)$, $\vy'=(\vy_{\parallel}',L+\eta')$
with  $\vy_{\parallel} \neq \vy'_{\parallel}$, the matrix $\widetilde{\vH}_r(\vy,\vy';k)$
defined by \eqref{eq.defhraproxs} has the asymptotic expansion
\begin{equation}\label{eq.decreascrossrange1}
\|\widetilde{\vH}_r(\vy,\vy';k)\|=  \frac{a^2}{L^2}\, \cO\Big(   \frac{ L}{  a \, k |\vy_{\parallel}-\vy'_{\parallel}
 | } \Big),
\end{equation}
in the Fraunhofer regime and for any matrix norm. The constant in the
$\cO$ notation depends only on the shape of the array $\cA$.
For $\vy'=\vy$, one has
\begin{equation}\label{eq.asymptappoxHr}
\widetilde{\vH}_r(\vy,\vy;k)= \frac{  \operatorname{mes}\cA}{ (4\pi L)^2} \vP_{\parallel},
\end{equation}
where  $\operatorname{mes} \cA=\Theta(a^2)$ is the array area.
\end{lemma}
\Cref{lem.asymHr} is proved in \cref{sec:prooflem1}.

\begin{proof}[Proof of \cref{prop.crossrang}]
The formulas \eqref{eq.asymp1} and \eqref{eq.asymp2} are a
straightforward consequence of \cref{prop.asymKI}, \cref{lem.asymHr} and
the fact that $\widetilde{\vH}_s(\vy_*,\vy_*;k)=\vP_s$.
\end{proof}

\begin{remark}
In the particular case where the receiver array $\cA$ is a square of
side $a$, one has $\operatorname{mes}\cA=a^2$ and by evaluating the
Fourier transform of the indicator function $\mathds{1}_{\cA}$, one
finds the following explicit expression of the matrix
$\widetilde{\vH}_r(\vy,\vy';k)$:
\begin{equation}\label{eq.sinccrossrange}
\widetilde{\vH}_r(\vy,\vy';k)= a^2 \frac{\exp[\imath k(\eta'-\eta)]}{(4\pi
L)^2} \operatorname{sinc}\Big(\frac{k\, a (y_1 -y'_1)}{2 \, L}\Big)
\operatorname{sinc}\Big(\frac{k \,a (y_2 - y'_2)}{2  L}\Big) \vP_{\parallel},
 \end{equation}
with $\vy_{\parallel}=(y_1,y_2)$ and $\vy'_{\parallel}=(y'_1,y'_2)$. 
\end{remark}
We now consider the problem of recovering the polarizability tensors
$\bGa_i$ of the $N$ dipoles. As in \cite{Cassier:2017:IPD}, the idea is
to observe that at each dipole position $\vy_i$ we have
\begin{equation}\label{eq.systmpola}
 \vH_r\left(\vy_i,\vy_i;k\right) \, \bGa_i \,
 \vH_s\left(\vy_i,\vy_i;k\right)=\cI_{KM}(\vy_i;k)-\sum_{j\neq i}  \vH_r\left(\vy_i,\vy_j;k\right) \, \bGa_j \,  \vH_s\left(\vy_i,\vy_j;k\right)^{\top}.
\end{equation}
If the dipoles are well-separated, \cref{lem.asymHr} guarantees that the
terms in \eqref{eq.systmpola} involving the dipoles $j\neq i$ remain small.
Thus, we expect a good estimate of the polarizability tensor $\bGa_i$ by
solving the linear system
\begin{equation}\label{eq.systpolaasymp}
\widetilde{\vH}_r\left(\vy,\vy;k\right)  \, \bGa \,
\widetilde{\vH}_s\left(\vy,\vy;k\right)=\cI_{KM}(\vy;k),
\end{equation}
for each point $\vy$ of the imaging window. Notice that
we replaced the matrices $\vH_r$ and $\vH_s$ by their Fraunhofer regime
approximations  $\widetilde{\vH}_r\left(\vy,\vy;k\right)$ and
$\widetilde{\vH}_s\left(\vy,\vy;k\right)$. Up to scaling factors, these latter matrices are
close to the rank two projectors
$\vP_{\parallel}$ and $\vP_s$, respectively. Therefore we cannot  expect
to retrieve the full polarizability tensor $\bGa$ in the Fraunhofer
regime. At most, we expect to recover only the tensor $\bGa$ projected
on the range of $\vP_{\parallel}$ (on the left) and the range of $\vP_s$
(on the right), that is the $2\times2$ matrix $\widetilde{\bGa}
=\vU_{\parallel}^{*} \, \bGa  \,  \vU_{s}$.  Indeed,  using
\eqref{eq.Hsasymp} for $\vy'=\vy$, \eqref{eq.asymptappoxHr} and
\eqref{eq.projectdef}, we rewrite \eqref{eq.systpolaasymp} as:
\begin{equation}
\frac{\operatorname{mes}\cA }{(4\pi L)^2} \,  \vU_{\parallel}\,
\vU_{\parallel}^{*} \, \bGa  \, \frac{1}{(4\pi L)^2}  \vU_{s}\,
\vU_{s}^{*}= \cI_{KM}(\vy;k).
\end{equation}
Multiplying on the left by $(L^2/ \operatorname{mes}\cA ) \,
\vU_{\parallel}^{*}$ and on the right by $(4\pi L)^2 \vU_{s}$ gives
\begin{equation}\label{eq.polareconst}
\widetilde{\bGa} = \frac{(4\pi L)^4}{ \operatorname{mes}\cA} \,
\widetilde{\cI}_{KM}(\vy;k),
\end{equation}
where $\widetilde{\cI}_{KM}(\vy;k)$ is also defined by
$\widetilde{\cI}_{KM}(\vy;k) =\vU_{\parallel}^{*} \,  \cI_{KM}(\vy;k) \,  \vU_{s}$.
In the following theorem, the Frobenius matrix norm is denoted by
$\|\cdot\|_F$.

\begin{theorem}(Cross-range estimation of the polarizability tensor)\label{prop.polacrossactiv}
Let  $\widetilde{\bGa}$ and $\widetilde{\bGa}_i$  be the $2\times 2$
matrices defined by $\widetilde{\bGa} =\vU_{\parallel}^{*} \, \bGa  \,
\vU_{s}$  and  $\bGa_i =\vU_{\parallel}^{*} \,   \, \bGa_i  \vU_{s}$.
Then we have the following
\begin{itemize}
\item If the imaging point coincides with the dipole, i.e. $\vy=\vy_i$,
\begin{equation}\label{eq.crossrangeplarecover}
\|\widetilde{\bGa}- \widetilde{\bGa}_i\|_F= \cO \Big(  \frac{L} {a\, k \min \limits_{j\neq
i}|\vy_{i,\parallel}-\vy_{j,\parallel}|} \Big)+  \cO\left(\frac{ a}{L}\right)+ \cO\left(\frac{a^2\Theta_a}{L^2}\right)+ \cO\left(\Theta_b \right).
\end{equation}
\item If the imaging point does not coincide with a dipole in cross-range,
i.e. if $\vy\neq \vy_j$ for all $j=1,\ldots,N$,
\begin{equation}\label{eq.drecraspola}
\| \widetilde{\bGa}\|_F = \cO \Big(  \frac{L} {a\, k \min
\limits_{j=1,\ldots,N}|\vy_{\parallel}-\vy_{j,\parallel}|} \Big)+ \cO\left(\frac{a^2\Theta_a}{L^2}\right)+ \cO\left(\Theta_b \right).
\end{equation}
\end{itemize}
\end{theorem}
\begin{proof}
We first prove the asymptotic relation \eqref{eq.crossrangeplarecover} when $\vy=\vy_i$. One checks easily by first left and right multiplying  both sides of  \eqref{eq.systmpola} respectively by $((4 \pi L)^2/ \operatorname{mes}\cA ) \, \vU_{\parallel}^{*}$ and  $(4\pi L)^2 \vU_{s}$ and then using the asymptotic formula \eqref{eq.spreadfunctionarray}  and \eqref{eq.asympspreadfunctionsource}  that
\begin{equation*}
\widetilde{\bGa}_i= \frac{ (4\pi L)^4}{ \operatorname{mes}\cA} \Big[ \widetilde{\cI}(\vy_i;k) -\sum_{j\neq i} \vU_{\parallel}^{*} \vH_r\left(\vy_i,\vy_j;k\right) \, \bGa_j \,  \vH_s\left(\vy_i,\vy_j;k\right)^{\top} \vU_{s}+ \frac{1}{ L^2 }o\Big(\frac{a^2}{L^2}\Big) \Big],
\end{equation*}
where $o(a^2/L^2)$ is explicitly given by
$\cO(a^4\Theta_a/L^4)+\cO(a^2\Theta_b/L^2)+\cO(a^3/L^3)$.
Thus it  follows immediately from \eqref{eq.polareconst} that
\begin{equation}\label{eq.singlefreqestimate}
 \widetilde{\bGa}-\widetilde{\bGa}_i=\frac{  (4\pi L)^4}{
 \operatorname{mes}\cA} \Big[ \sum_{j\neq i} \vU_{\parallel}^{*}
 \vH_r\left(\vy_i,\vy_j;k\right) \, \bGa_j\,
 \vH_s\left(\vy_i,\vy_j;k\right)^{\top} \vU_{s}+ \frac{1}{ L^2
 }o\Big(\frac{a^2}{L^2}\Big) \Big],
\end{equation}
where
$o(a^2/L^2)=\cO(a^4\Theta_a/L^4)+\cO(a^2\Theta_b/L^2)+\cO(a^3/L^3)$.  \\
\noindent Finally we conclude by using that
$\operatorname{mes}\cA = \Theta(a^2)$,
 $\|\vU_{\parallel}^{*}\|_F=\|\vU_s\|_F=\sqrt{2}$, and the
asymptotic expansion \eqref{eq.asymp1} to control the coupling terms
\[
\|\sum_{j\neq i}\vH_r\left(\vy_i,\vy_j;k\right) \, \bGa_j \,
\vH_s\left(\vy_i,\vy_j;k\right)^{\top}\|_F.\]
 
The asymptotic expansion \eqref{eq.drecraspola} where $\vy\neq \vy_i$ is
an immediate consequence of the reconstruction formula
\eqref{eq.polareconst}, $\operatorname{mes}\cA = \Theta(a^2)$,
\cref{prop.crossrang} which expresses the decay of the Kirchhoff image
$\| \cI_{KM}(\vy;k) \|_F$ in the case of one dipole and the linearity of
$\cI_{KM}(\vy;k)$ with respect to the data.
\end{proof}

\Cref{prop.polacrossactiv} and estimate \eqref{eq.crossrangeplarecover}
guarantee a good estimate of $\widetilde{\bGa}_i$ using the
reconstruction formula \eqref{eq.polareconst} at the dipole position
$\vy_i$, provided the cross-range distance to the closest dipole is
large compared to the Rayleigh criterion $L/(k a)$. The other error
terms vanish in the Fraunhofer asymptotic regime. The second estimate
\eqref{eq.drecraspola}  shows that $\|\widetilde{\bGa}\|_F$ at an
imaging point $\vy$ decays as the inverse of the cross-range distance to
the closest dipole. Thus, $\|\widetilde{\bGa}\|_F$ is also a good
imaging function for the cross-range position of the dipoles.  Indeed,
it has the same cross-range resolution $L/(ak)$  (see
\eqref{eq.drecraspola}) as the Kirchhoff imaging function
$\cI_{KM}(\vy;k)$ (see \cref{prop.crossrang}).

\subsection{Estimates of range position and polarizability tensor}\label{sec.rangeesti}
We now assume that the data \eqref{eq::data1} is available over a
frequency band $[\omega_0-B/2, \omega_0+B/2]$ of central (angular)
frequency $\omega_0$ and bandwidth $B$. The Kirchhoff image of the
multi-frequency data is given by integrating over the frequency band the
single frequency Kirchhoff imaging function
\eqref{eq.defimagingfuction}, that is
\begin{equation}\label{eq.Kirchmf}
\cI_{KM}(\vy)=\int_{|\omega-\omega_0|<B/2} d \omega \, \cI_{KM}\left(\vy;\frac{\omega}{c}\right).
\end{equation}
We further assume that the dipoles are standard scatterers, i.e. their
rescaled polarizability tensor $\bGa(\vy,\omega)=\mu \omega^2
\bGa'(\vy,\omega)$ is constant with respect to the frequency on the
frequency bandwidth $B$. This a usually a good approximation for
dielectrics or metals at high frequency (see e.g.
\cite{Novotny:2012:PNO}) where the ``physical'' polarizability tensor is
shown to be $\bGa'(\vy,\omega) = \cO(\omega^{-2})$ as $\omega\to \infty$.

We now estimate the range resolution of the imaging function
\eqref{eq.Kirchmf}. Without loss of generality, we consider the case of
a single dipole $\vy_*=(\vy_{*,\parallel},L+\eta_*)$ whose cross-range
position $\vy_{*,\parallel}$ is known. For the analysis and in the
following, we assume that all the asymptotic expansions of \cref{sec.asympparam} 
hold uniformly with $k=\omega/c$ in the frequency band.  

Let $\vy=(\vy_{*,\parallel},L+\eta)$ be an imaging point. With
\cref{prop.asymKI} we see that the imaging function \eqref{eq.Kirchmf} satisfies
\begin{equation*}
\cI_{KM}(\vy) = \frac{1}{(4\pi L)^2}\int_{|\omega-\omega_0|<B/2} d
\omega \widetilde{\vH}_r\Big(\vy,\vy_*;\frac{\omega}{c}\Big) \, \bGa_{*}
\widetilde{\vH}_s\Big(\vy,\vy_*;\frac{\omega}{c}\Big)^{\top} + \frac{B}{(4\pi L)^2} \,
o\Big(\frac{a^2}{L^2}\Big),
\end{equation*}
where the $o(a^2/L^2)$ term can be explicitly given  by
$\cO\left(a^4\Theta_a /L^4\right)+\cO( a^2\Theta_b /L^2)+ \cO\left(a^3
/L^3\right)$ (see \cref{prop.asymKI}). Hence we get that
\begin{equation}\label{eq.dipolealigninrange}
\cI_{KM}(\vy) = \frac{  \operatorname{mes}\cA } {(4\pi L)^4} \,
\vP_{\parallel}  \bGa_* \, \vP_s  \int_{|\omega-\omega_0|<B/2} d \omega
\, \exp[\imath \omega\, \big(\phi(\eta_*)-\phi(\eta)\big)/c] +   \frac{B}{(4\pi L)^2} \, o\Big(\frac{a^2}{L^2}\Big),
\end{equation}
where $\phi$ is the smooth phase function 
\[
\phi(\eta)= \eta+|\vx_s-\vy|
=\eta+\sqrt{|\vy_{\parallel}-\vx_{s,\parallel}|^2+(L+\eta-x_{s,3})^2}.
\]
The integral over the frequency band can then be evaluated explicitly to
obtain
\begin{equation}\label{eq.asymmfrange}
\begin{aligned}
\cI_{KM}(\vy) &=\frac{ B \operatorname{mes}\cA } {(4\pi L)^4}
\vP_{\parallel}  \bGa_* \, \vP_s \exp[\imath \omega_0
\big(\phi(\eta_*)-\phi(\eta)\big)/c ] \, \operatorname{sinc}\Big(\frac{B
\big(\phi(\eta_*)-\phi(\eta)\big)}{2 \, c}\Big)\\
&+ \frac{B}{(4\pi L)^2} o\left(\frac{a^2}{L^2}\right).
\end{aligned}
\end{equation}

The next proposition shows that as soon as the range distance
$|\eta - \eta_*|$  between the imaging point and the dipole becomes
large compared to $c/B$, the norm of the Kirchhoff imaging function
\eqref{eq.asympbf1} becomes small compared to its value at the dipole
position \eqref{eq.asympbf2}. Thus the range resolution is given by
$c/B$ (as in acoustics \cite{Blei:2013:MSI,Borcea:2007:OWD}  or in
electromagnetics with collocated sources and receivers
\cite{Cassier:2017:IPD}).
\begin{proposition}[Range resolution of the position]
\label{prop.rangeresolution}
Assume the geometric condition \eqref{eq:geoimg} holds for the source.
The Kirchhoff imaging function \eqref{eq.Kirchmf} of a dipole located at
$\vy_*=(\vy_{*,\parallel}, L+\eta_*)$ and evaluated at
$\vy=(\vy_{*,\parallel},
L+\eta)$ satisfies 
\begin{equation}\label{eq.asympbf1}
 \|\cI_{KM}(\vy;k)\| =  \frac{B}{(4\pi L)^2}  \frac{ a^2}{L^2} \left(  \cO\Big(    \frac{ c}{
 B   \left|\eta-\eta_{*} \right|} \Big)+o(1)\right),~\text{when}~\eta
 \neq \eta_*,
\end{equation}
for any matrix norm. The remainder $o(1)$ term can be explicitly given
by $\cO\left(a^2\Theta_a /L^2\right)+\cO( \Theta_b)+
\cO\left(b/L\right)$.
When $\eta = \eta_*$, one has 
\begin{equation}\label{eq.asympbf2}
\cI_{KM}(\vy_*;k)=  \ \frac{B}{(4\pi L)^2} \Big[\frac{\operatorname{mes}\cA }{(4\pi L)^2} \vP_{\parallel} \bGa_* \, \vP_s+ \, o\Big(\frac{a^2}{L^2}\Big)\Big],
\end{equation}
where  $\operatorname{mes}  \cA = \Theta(a^2)$ is the area of the array.
\end{proposition}
\begin{proof}
Since $\phi$ is a smooth function, the mean value theorem implies that
\begin{equation}\label{eq.taylorphi}
|(\phi(\eta)-\phi(\eta_*))^{-1}|\leq \max_{\vy'\in
\cW}|\phi'(\eta')^{-1}|\, |\eta-\eta_*|^{-1}=\cO(|\eta-\eta_*|^{-1}),
\end{equation}
where $\vy'=(\vy_{\parallel}',L+\eta')$ is a point belonging to the
imaging window $\cW$. In \eqref{eq.taylorphi} we can use
$\phi'(\eta')^{-1}$ because
\begin{equation}\label{eq.derivphase}
\phi'(\eta')=1+\frac{L+\eta'-x_{s,3}}{|\vx_s-\vy'|}\neq 0
\end{equation}
since $\vy_{\parallel}' \neq \vx_{s,\parallel}$. This last condition on
the cross-range of the source is weaker than the geometric condition
imposed in \eqref{eq:geoimg} since the characteristic size $b$ of the
imaging widow is small compared to characteristic size $a$ of the array
of receivers. As $\cW$ is a compact set, this implies that
$\max_{\vy'\in \cW}|(\phi'(\eta'))^{-1}|$ exists in
\eqref{eq.taylorphi}.  Thus, it follows form \eqref{eq.taylorphi} that
\begin{equation}
\operatorname{sinc}\Big(\frac{B \big(\phi(\eta_*)-\phi(\eta)\big)}{2 \, c}\Big) = \cO\Big(    \frac{ c}{B   \left|\eta-\eta_{*} \right|} \Big).
\label{eq:sincphi}
\end{equation}
Thus, the asymptotic expansion\eqref{eq.asympbf1} is an immediate
consequence of \eqref{eq.asymmfrange},  \eqref{eq:sincphi} and the fact
that $\operatorname{mes}\cA = \cO(a^2)$.  Finally, the asymptotic
relation \eqref{eq.asympbf2} follows immediately by evaluating
\eqref{eq.dipolealigninrange} at $\vy=\vy_*$.
\end{proof}

Since we assume the polarizability tensor $\bGa$ is frequency
independent (see \cref{sec:problem:setup}), we suggest to estimate it by
averaging the single frequency estimate \eqref{eq.crossrangeplarecover}
over the frequency band
\begin{equation}\label{eq.estimatemf}
\widetilde{\bGa} = \frac{1}{B}\int_{\cB} d\omega
\, \widetilde{\bGa}(\omega).
\end{equation}
Note that \eqref{eq.estimatemf} involves the projected polarizability
tensor $\widetilde{\bGa} = \vU_{\parallel}^* \bGa \vU_s$ (which is
estimated in \eqref{eq.polareconst} from the Kirchhoff image), and 
the matrices $\vU_{\parallel}^{*}$ and $\vU_s$ are frequency
independent. Hence it is straightforward to check that integrating the
single frequency estimate \eqref{eq.crossrangeplarecover} and \eqref{eq.drecraspola} with
\eqref{eq.estimatemf} over the frequency band does not change these
cross-range resolution estimates.

Here we study the effect of the other dipoles in the polarizability
tensor estimate in range. We isolate the effect of range by considering
the case where all the dipoles have same cross-range, i.e.
$\vy_i=(\vy_{*,\parallel},L+\eta_i)$ for $i=1, \ldots, N$. The following
proposition shows that the depth resolution of the reconstructed
$\widetilde{\bGa}$ is also $c/B$. Concretely, the estimate
\eqref{eq.rangepolaretensorcover1} shows that one has a good depth
resolution of $\widetilde{\bGa}_i$ by the reconstruction formula
\eqref{eq.estimatemf} at the dipole position $\vy_i$, as soon as the
range distance of the different dipoles is large compared to $c/B$. The
second estimate \eqref{eq.drecraspolatensprrange2}  shows that
$\|\widetilde{\bGa}\|_F$ at an imaging point $\vy$ decays when the range
distance to the closest dipole becomes large compared to $c/B$. Thus,
$\|\widetilde{\bGa}\|_F$ is also a good imaging function for  the range
position of a dipole  with  the same depth resolution $c/B$ (see
\eqref{eq.drecraspolatensprrange2}) as the Kirchhoff imaging function (see
\cref{prop.rangeresolution}).
\begin{theorem}\label{prop.polarange.mat}
Assume the geometric condition \eqref{eq:geoimg} holds for the source
and the dipoles are all aligned  in the range direction of $\cA$. The
image of the cross-range polarizability tensor $\widetilde{\bGa}$ (given
by \eqref{eq.estimatemf}) satisfies the two following estimates:
\begin{itemize}
\item If the imaging point is the dipole location, i.e. $\vy=\vy_i$,
we have
\begin{equation}\label{eq.rangepolaretensorcover1}
\|\widetilde{\bGa}- \widetilde{\bGa}_i\|_F= \cO \Big(  \frac{c} {B \min \limits_{j\neq i}|\eta_i-\eta_j|}\Big)+o(1),
\end{equation}
\item If the imaging point range is different from any of the dipole
ranges, $\vy=(\vy_{*,\parallel},L+\eta)\neq \vy_j$, for all $j=1,\ldots,N$,
\begin{equation}\label{eq.drecraspolatensprrange2}
\|\widetilde{\bGa}\|_F =   \cO \Big(  \frac{c} {B \min \limits_{j=1,\ldots,N}|\eta-\eta_j|} \Big)+o(1).
\end{equation}
\end{itemize}
\end{theorem}
\begin{proof}
We first show the asymptotic expansion
\eqref{eq.rangepolaretensorcover1}, i.e. when $\vy=\vy_i$. Since
$\bGa_i$, $\vU_\parallel$ and $\vU_s$ are frequency independent, one has 
\begin{equation}
\widetilde{\bGa}- \widetilde{\bGa}_i= \frac{1}{B}\int_{B} d \omega  \,
(\widetilde{\bGa}(\omega)- \widetilde{\bGa}_i).
 \end{equation}
Hence \eqref{eq.singlefreqestimate} and the previous expression gives that
 \begin{equation*}
 \widetilde{\bGa}- \widetilde{\bGa}_i= \frac{  (4\pi L)^4}{B \operatorname{mes}\cA}  \int_{B} \,  \Big[ \vU_{\parallel}^{*}\,\big(\sum_{j\neq i}\vH_r\left(\vy_i,\vy_j;k\right) \, \bGa_j \,  \vH_s\left(\vy_i,\vy_j;k\right)^{\top} \big)\, \vU_s +o\Big(\frac{a^2}{L^2}\Big) \Big].
 \end{equation*}
Now recall $(\operatorname{mes}\cA)^{-1} = \cO(a^{-2})$. Using the
expression \eqref{eq.asymmfrange} of the imaging function for aligned
dipoles we get 
 \begin{equation*}
 \widetilde{\bGa}- \widetilde{\bGa}_i=
  \sum_{j\neq i}\widetilde{ \bGa}_j \exp[\imath \omega_0
  \big(\phi(\eta_j)-\phi(\eta_i)\big)/c ] \,
  \operatorname{sinc}\left(\frac{B
  \big(\phi(\eta_j)-\phi(\eta_i)\big)}{2 \, c}\right) +o(1).
  \end{equation*}
We arrive to \eqref{eq.rangepolaretensorcover1} by using the asymptotic
relation \eqref{eq.taylorphi} (which holds under the geometric
condition) to control the $\operatorname{sinc}$ and the fact that
$\widetilde{ \bGa}_j =\cO(1)$.  Finally, the asymptotic
formula \eqref{eq.drecraspolatensprrange2} can be proved from the
relation \eqref{eq.estimatemf} in a similar way.
\end{proof}

\subsection{Correction of oscillatory artifacts}\label{sec.phasecorrection}
In acoustics, it is a well-known that the reflection coefficient of a
point scatterer can only be recovered up to a complex phase, see e.g.
\cite{Novikov:2015:ISI}. A similar phenomenon is observed in
electromagnetism.  For the case of collocated sources and receivers
\cite{Cassier:2017:IPD}, the polarizability tensor of a dipole is for
all practical purposes, recovered only up to a complex phase because the
Kirchhoff imaging function oscillates in range. These oscillatory
artifacts can be corrected by fixing the phase of one of the components
of the estimated polarizability tensor \cite{Cassier:2017:IPD}.
In the non-collocated sources and receivers case, we observe
oscillatory artifacts in range (\cref{fig:r:stable}) and also in
cross-range (\cref{fig:cr:stable}). Here we explain where these
oscillatory artifacts come from and why they also occur in
cross-range.  Both oscillatory artifacts can be corrected in the same
manner as in \cite{Cassier:2017:IPD}, as is illustrated in
\cref{fig:cr:stable,fig:r:stable}.

To observe the oscillations in the range direction, we consider the
reconstruction formula \eqref{eq.estimatemf} in the case of a single
dipole ($N=1$) located at $\vy_*$ with polarizability tensor $\bGa_*$,
and for imaging points of the form $\vy=(\vy_{*,\parallel},L+\eta)$.
Using \eqref{eq.polareconst}, \cref{prop.asymKI} with the
relations \eqref{eq.defhraproxs}, \eqref{eq.Hsasymp}  and integrating
over the bandwidth, we rewrite \eqref{eq.estimatemf} as
\begin{eqnarray*}
\widetilde{\bGa}&=& \int_{|\omega-\omega_0|<B/2} d \omega  \,
\exp\Big[\imath \frac{\omega}{c}\, (\phi(\eta_*)-\phi(\eta))\Big]
\widetilde{\bGa}_*  +o(1)\\ &=&   \exp\left[\imath  \frac{\omega_0}{c}
(\phi(\eta_*)-\phi(\eta))\right]\operatorname{sinc}\Big(\frac{B
(\phi(\eta_*)-\phi(\eta))}{c}\Big) \widetilde{\bGa}_*+o(1).
\end{eqnarray*}
Under the geometric condition \eqref{eq:geoimg}, we have
$\phi(\eta_*)-\phi(\eta)=\phi'(\eta_*)(\eta_*-\eta)+o(\eta-\eta_*)$ with
$0<\phi'(\eta_*)<2 $ (see \eqref{eq.derivphase}), thus the presence of
the complex exponential  $\exp[\imath \omega_0\,
(\phi(\eta_*)-\phi(\eta)/c)]$ and the $\operatorname{sinc}$ causes the
image of 
$\widetilde{\bGa}$ to oscillate in $\eta$ around the dipole range
position $\eta_*$.  Compared to the case of collocated sources and
receivers  \cite{Cassier:2017:IPD}, where the function $\phi$ is replaced
by $\eta$, the presence of
the factor $\phi'(\eta_*)$ in the Taylor expansion of $\phi$ around
$\eta_*$ takes into account the relative positions of the source, the
array and the dipole. Furthermore, if the source range  $x_{s,3}$
is between the array and the imaging points, i.e. if $0<x_{s,3}<L+\eta$
for all imaging points, one has by \eqref{eq.derivphase} that
$1<\phi'(\eta_*)<2$. Thus, if the central angular frequency $\omega_0$
and the bandwidth $B$ are of the same order, {\em the oscillations'
length-scale is $c/B$, the focal spot size  in depth}. 
Therefore, if the range position $L+\eta_*$ of the dipole is not known
precisely, we cannot expect to accurately reconstruct
$\widetilde{\bGa}_*$.

We deal with this artifact by fixing the phase of one component of the
$2\times 2$ matrix $\widetilde{\bGa}$, as in \cite{Cassier:2017:IPD}.
The choice we made here is to enforce that the $1,1$ entry be real and
positive, i.e.  $\arg(\widetilde{\alpha}_{1,1}) =0$. This can be achieved
by post-processing the reconstruction formula \eqref{eq.estimatemf} by
the operation
$(\overline{\widetilde{\alpha}_{1,1}}/|\widetilde{\alpha}_{1,1}|) \,
\widetilde{\bGa}$.  If the coefficient $|\widetilde{\alpha}_{1,1}|$ is small, this
operation can be problematic. If necessary, the phase of another
coefficient of $\widetilde{\bGa}$ can be fixed or we can regularize 
by using
$(\overline{\widetilde{\alpha}_{1,1}}/(|\widetilde{\alpha}_{1,1}|+\delta))
\widetilde{\bGa}$, for a small $\delta>0$.

To explain the oscillations in the cross-range direction, we evaluate
the single frequency reconstruction formula \eqref{eq.polareconst} using
\cref{prop.asymKI}, \eqref{eq.defhraproxs} and \eqref{eq.Hsasymp}  at
image points $\vy=(\vy_{\parallel},L+\eta_*)$ in the cross-range of the
dipole $\vy_*$ to get 
\begin{equation}\label{eq.oscillcrossrange}
\widetilde{\bGa}= \exp \left[ \imath k
[\psi(\vy_{*,\parallel})-\psi(\vy_{\parallel})]\right]
\mathcal{F}[\mathds{1}_{\cA}]\Big(\frac{k}{L}(\vy_{*,\parallel}-\vy_{\parallel})\Big)
\widetilde{\bGa}_* +o(1),
\end{equation}
where the $o(1)$ remainder can be given explicitly by
$\cO\left(a^2\Theta_a /L^2\right)+\cO( \Theta_b)+ \cO\left(a/L\right)$
and the  smooth function $\psi$ is defined by
$\psi(\vy_{\parallel})=|\vx_s-\vy|$. To give an explicit
expression of these oscillations, we assume that $\cA$ is a square of
side $a$. Hence \eqref{eq.oscillcrossrange} can be computed explicitly 
with the Fourier transform of $\mathds{1}_{\cA}$ (as in
\eqref{eq.sinccrossrange}):
\[
\widetilde{\bGa}= \exp \Big[ \imath k
[\psi(\vy_{*,\parallel})-\psi(\vy_{\parallel})]\Big]
\operatorname{sinc}\Big(\frac{k\,a (\vy_{*,1} -\vy_{1})}{2  L}\Big)
\operatorname{sinc}\Big(\frac{k\, a (\vy_{*,2} -\vy_{2})}{2  L}\Big)
\widetilde{\bGa}_* +o(1).  \]
In the last formula, we have
\[
\psi(\vy_{*,\parallel})-\psi(\vy_{\parallel})= \nabla
\psi_{\vy_{*,\parallel}}\cdot
(\vy_{*,\parallel}-\vy_{\parallel})+o(|\vy_{*,\parallel}-\vy_{\parallel}|)
\mbox{ with }\nabla \psi_{\vy_{*,\parallel}}=
\frac{\vx_{s,\parallel}-\vy_{*,\parallel}}{|\vx_s-\vy_*|}. 
\]
Notice that $\nabla \psi_{\vy_{*,\parallel}}\neq 0$ since $\vx_{s,\parallel}\neq
\vy_{*,\parallel}$ under the geometrical condition \eqref{eq:geoimg}. 
On one hand, the presence of the $\operatorname{sinc}$ terms induces an
oscillation whose length scale is given by the Rayleigh criterion
$ka/L$, the cross-range focal spot size. On the
other hand, the complex exponential $\exp \left[ \imath k
[\psi(\vy_{*,\parallel})-\psi(\vy_{\parallel})]\right]$ induces an
oscillation whose length scale is controlled in each cross-range
direction (at the vicinity of $\vy_*$) by $2\pi /(k |\nabla
\psi_{\vy_{*,\parallel}}|)$ and thus could rapidly change with
$|\vx_s-\vy_*|$. With these oscillations,
one cannot expect to reconstruct $\widetilde{\bGa}$ accurately in
cross-range. Of course, the same post-processing operation
$(\overline{\alpha_{1,1}}/|\widetilde{\alpha}_{1,1}|) \,
\widetilde{\bGa}$ eliminates the oscillation artifacts in both range and
cross-range images of $\widetilde{\bGa}$. This is illustrated by the
numerical experiments in \cref{fig:cr:stable,fig:r:stable}. 
\subsection{Physical measurements}
\label{sec.physmeas}
We now show that the full data $\vPi$ is not necessary to 
obtain Kirchhoff images. We consider here the case where we only have
access to the projected data $\vP_{\parallel} \vPi
\vP_s=\vU_{\parallel}\, \widetilde{\vPi}\vU_s^{*}$ (where
$\widetilde{\vPi}$ defined by \eqref{eq:GPi:coeff}). This is a more
physically relevant setup as it assumes only the electric field
components parallel to the array can be measured and that the source
polarization can only be controlled in the two dimensional subspace
$(\vx_s - \vy_0)^\perp$ (see \cref{sec:poldata,sec:strategy}). To this
end, we show that the Kirchhoff imaging functions
$\cI_{KM}[\vP_{\parallel} \vPi \vP_s]$ and $\cI_{KM}[\vPi]$ have the same cross-range and range resolutions. In other words the additional information
contained in the full data problem does not bring any new information
about the position, at least in the Fraunhofer regime. This can be
expected because the wave emanating from the source is close to a plane
wave in the vicinity of the scatterer, making one component of the
electric field irrelevant. Similarly the scattered wave is a
close to a plane wave near the array (see \cref{sec:poldata}).
\begin{proposition}\label{prop.incompletedata}
The imaging function $\cI_{KM}[\vP_{\parallel} \vPi \vP_s](\cdot;k)$
satisfies the asymptotic relation \eqref{eq.asymp1} and the asymptotic
relation \eqref{eq.asymp2}. The multi-frequency imaging function
$\cI_{KM}[\vP_{\parallel} \vPi \vP_s](\cdot)$  satisfies the asymptotic
relations \eqref{eq.asympbf1} and \eqref{eq.asympbf2} in range. Moreover
using $\cI_{KM}[\vP_{\parallel} \vPi \vP_s](\cdot;k)$ to reconstruct
$\widetilde{\bGa}$ over the bandwidth (formulas \eqref{eq.estimatemf} and \eqref{eq.polareconst}) leads to the same results in range and cross-range. In other words 
the propositions \ref{prop.polacrossactiv} and
\ref{prop.polarange.mat} hold with the partial data $\vP_{\parallel}
\vPi \vP_s$.
\end{proposition}
\begin{proof}
Using the linearity of the Kirchhoff imaging function with
respect to the data, we consider the case of a single dipole located at
$\vy_*=(\vy_{*,\parallel},L+\eta_*)$ with polarizability tensor $\bGa_*$.
For $\vy=(\vy_\parallel,L+\eta_*)$ and partial data $\vP_{\parallel} \vPi
\vP_s$ the image is
\begin{equation}\label{eq.partialdata1}
\begin{aligned}
\cI_{KM}&[\vP_{\parallel} \vPi \vP_s](\vy, k)\\
&=\int_{\cA}d\vx_{r,\parallel}  \,
\overline{\vG\left(\vx_r, \vy;k\right)} \,  \vP_{\parallel}
\vG\left(\vx_r,\vy_*;k\right) \, \bGa(\vy_*)  \vG\left(\vy_*,\vx_s;k\right) \, \vP_s
\overline{\vG\left(\vx_s,\vy; k\right)}.   
\end{aligned}
\end{equation}
Using the asymptotic expression of the dyadic Green function
\eqref{eq.asymptgreenarray}, we can show in the same way as for the
relation  \eqref{eq.spreadfunctionarray} that
\begin{equation}\label{eq.spreadfunctionarrayincompletedata}
\int_{\cA}d\vx_{r,\parallel}  \,
\overline{\vG\left(\vx_r, \vy;k\right)} \,  \vP_{\parallel}
\vG\left(\vx_r,\vy_*;k\right) =\widetilde{\vH}_r\left(\vy,\vy_*;k\right)
+ \cO\Big(\frac{a^4\Theta_a}{L^4}\Big)+ \cO\Big(\frac{a^2\Theta_b}{L^2} \Big)+\cO\Big(\frac{a^3}{L^3} \Big).
\end{equation}
We now consider the term linked to the source in
\eqref{eq.partialdata1}. By applying the asymptotics relation
\eqref{eq.aproxsource}, we can show as in \eqref{eq.asympspreadfunctionsource} that 
\begin{equation}\label{eq.partialdatasource}
 \vG\left(\vy_*,\vx_s;k\right) \, \vP_s
\overline{\vG\left(\vx_s,\vy; k\right)}  = \frac{1}{(4\pi L)^2} \Big[\widetilde{\vH}_s\left(\vy,\vy_*;k\right)^{\top}+O\Big(\frac{b}{L}\Big)+\cO\Big(\frac{a^2}{\Theta_a L^2}\Big)\Big].
\end{equation}
Thus, by combining \eqref{eq.spreadfunctionarrayincompletedata} and
\eqref{eq.partialdatasource}, we get
\[
\cI_{KM}[\vP_{\parallel}\vPi \vP_s](\vy, k)=  \frac{1}{(4\pi L)^2}  \widetilde{\vH}_r\left(\vy,\vy_*;k\right) \,
\bGa(\vy_*) \, \widetilde{\vH}_s\left(\vy,\vy_*;k\right)^{\top}+o(1),
\]
where the $o(1)$ remainder can be explicitly given by $\cO\left(a^4\Theta_a/L^4\right)+ \cO\left(a^2\Theta_b/L^2 \right)+\cO\left(a^3/L^3 \right)$.
The rest of the proof is identical to the proof of \cref{prop.crossrang}
since $\cI_{KM}[\vP_{\parallel} \vPi \vP_s](\vy;k)$ and $\cI_{KM}[\vPi](\vy; k)$
have exactly the same Fraunhofer asymptotic expansion. Thus, the
asymptotic expansions \eqref{eq.asympbf1} and \eqref{eq.asympbf2} in
range can be derived exactly in the same way as in \cref{sec.rangeesti}.
Finally, as these new data have the same asymptotic properties as
the full data, doing the same post-processing as in
\cref{sec.crossrangepola,sec.rangeesti} leads to an accurate reconstruction
of $\widetilde{\bGa}$, in both cross-range and range. To be more precise,
\cref{prop.polacrossactiv,prop.polarange.mat} can also be proved identically.

\end{proof}

\section{The coherency matrix from time domain electric field autocorrelations}\label{sec:stochillum}
Here we show that the coherency matrix data \eqref{eq:cohmat} can be
obtained from measuring the time domain electric field autocorrelations
at the array resulting from illuminating with a point source that is
driven by a random process (i.e. white light). To this end, the
assumptions \eqref{eq:js:assumption} need to be supplemented to account for
the frequency dependence of the polarization vector $\vj_s(\omega)$.
Indeed, since the time domain polarization vector $\vj_s(t)$ is real,
$\vj_s(\omega)$ obeys the reflection principle $\vj_s(-\omega) =
\overline{\vj_s(\omega)}$. Taking this into account, we must have that for
any $\omega, \omega' \in \real$
\begin{equation}
 \begin{aligned}
   \Ma{ \vj_s(\omega) } &= 0\\
   \Ma{\vj_s(\omega) \vj_s(\omega')^T} &= \delta(\omega + \omega') \vU_s
   \widetilde{\vJ}_s(\omega) \vU_s^*,
   ~\text{and}~\\
   \Ma{\vj_s(\omega) \vj_s(\omega')^*} &= \delta(\omega - \omega') \vU_s
   \widetilde{\vJ}_s(\omega) \vU_s^*,
   \label{eq:js:omega}
 \end{aligned}
\end{equation}
where $\widetilde{\vJ}_s(\omega)$ is a (known) $2\times 2$ Hermitian matrix with
$\widetilde{\vJ}_s(-\omega) = \overline{\widetilde{\vJ}_s(\omega)}$. With these assumptions it is easy to see that $\vj_s(t)$ is a real Gaussian stationary random vector satisfying
\begin{equation}
  \Ma{\vj_s(t)} = \vzero,~\text{and}~
  \Ma{\vj_s(t+\tau)\vj_s(t)^T} =  \vU_s \widetilde{\vJ}_s(\tau) \vU_s^*,
  \label{eq:js:t}
\end{equation}
where $\widetilde{\vJ}_s(\tau)$ is the Fourier transform (using the
convention \eqref{eq:fourier:convention}) of
$\widetilde{\vJ}_s(\omega)$. This is a version of the Wiener-Khinchin
theorem (see e.g. \cite{Ishimaru:1997:WPS}) for vectors.
The next proposition shows that when using a random source, measuring
empirical autocorrelations of the electric field at the array can be
used to find the coherency matrix data $\vpsi(\omega)$ in
\eqref{eq:cohmat}, provided the acquisition time is large.

\begin{proposition}[Statistical stability]\label{prop:statstab} Assume
$\vj_s(t)$ is a stationary Gaussian process satisfying
\eqref{eq:js:t} and the Born approximation holds. For each $\vx_r\in\cA$, the empirical
autocorrelations are
\begin{equation}\label{eq:empauto}
\vpsi_{emp}(\vx_r,\tau) = \frac{1}{2T}\int_{-T}^T dt
\vU_\parallel^* \vE(\vx_r,t+\tau)\vE(\vx_r,t)^T \vU_\parallel.
\end{equation}
Their mean is independent of the measurement time $T$ and is given by
\begin{equation}\label{eq:avgempauto}
\Ma{\vpsi_{emp}(\vx_r,\tau)} = (2\pi) \int d\omega
e^{-\i\omega\tau} \vpsi(\vx_r,\omega),
\end{equation}
Moreover, the empirical autocorrelations are ergodic:
\begin{equation}\label{eq:ergodic}
\vpsi_{emp}(\vx_r,\tau)
\rightarrow 
\Ma{\vpsi_{emp}(\vx_r,\tau)} ~\text{as}~ T \to \infty.
\end{equation}
\end{proposition}
\Cref{prop:statstab} is the electromagnetic analogue of a statistical
stability result for scalar waves that was proven originally in
\cite{Garnier:2009:PSI} (see also \cite{Bardsley:2016:IPC}) and is
proved in \Cref{sec:app_statstab}.

\section{Numerical experiments}
\label{sec:numerics}
We illustrate the proposed imaging method with numerical experiments in a regime corresponding to microwaves propagating in vacuum
(\cref{sec.physparam}). We start in \cref{sec.numdet} with experiments
where the data is the coherency matrix over a certain bandwidth (as
defined in \cref{sec:problem:setup}).  We refer to these experiments as
deterministic, because no empirical averaging is performed. We
investigate the behavior of the imaging routine for a few dipole
scatterers, as well as for an extended scatterer. The visualization of
the $2 \times 2$ projected polarizability tensors that appear in our
images is explained in \cref{sec.viz}. We then report in
\cref{sec.numstoc} an experiment with the electric dipole source being
driven by a time domain Gaussian process and the measured data are the
empirical correlations \eqref{eq:empauto}.

\subsection{Physical parameters}\label{sec.physparam}
In all of our experiments, we assume a homogeneous background medium
with wave speed given by that of light in vacuum, i.e., $c =
(\varepsilon\mu)^{-1/2} = 3\times 10^8$ ms$^{-1}$. We use the central
frequency $\omega_0/(2\pi) = 2.4$ GHz, which gives a central wavelength
of $\lambda_0 = 0.125$ m. Our receiver array consists of $61\times 61$
receivers located at the points $\cA = \{(u_i,v_j,0):
u_i=i(20\lambda_0/61), v_j=j(20\lambda_0/61), i,j=-30,\ldots,30\}.$ This
corresponds to a square array of side $a=20\lambda_0$ in the $x_3=0$
plane, centered at the origin, consisting of $61\times 61$ uniformly
spaced receivers. We consider a characteristic propagation distance of
$L=100\lambda_0$, and place a single point source at the location $\vx_s
= (L/2,0,L(1-\sqrt{3/4}))$.  This corresponds to a point source located
on the $x_2=0$ plane, at a distance $L$ from the reference point
$\vy_0=(0,0,L)$ and the source is outside of the region $\cR(3)$
defined in \eqref{eq:geoimg}. All of the scatterers we consider are located near
$\vy_0$, so we recover Kirchhoff images for points $\vy$ within the cube
of side $30\lambda_0$ centered at $\vy_0$, i.e.,
\[
\cW = \{\vy\in\real^3: \|\vy-\vy_0\|_\infty \leq 15\lambda_0\}.
\]

\subsection{Matrix visualization convention}\label{sec.viz}
Recall that the imaging method of \cref{sec:imaging} only recovers a
$2\times 2$ projection of the $3\times 3$ complex polarizability tensor.
The recovered polarizability tensor is not symmetric in general (it is
symmetric for the collocated sources and receivers case \cite{Cassier:2017:IPD}).
We visualize matrices $\vA \in \real^{2 \times 2}$ with the ellipse
$\cE(\vA) = \{ \vA \vv ~|~ \|\vv\|_2=1\}$. To emphasize that the
matrices we recover are not symmetric, we also display (when possible)
the vectors $\sigma_1 \vv_1$ and $\sigma_2 \vv_2$, where we have used
the singular values $\sigma_j$ and right singular vectors $\vv_j$,
$j=1,2$, from
the SVD $\vA = [\vu_1 \vu_2] \diag(\sigma_1,\sigma_2) [\vv_1^T
\vv_2^T]$. Indeed, if the matrix $\vA$ is symmetric then $\vv_j = \vu_j$
(up to a sign) so the vectors $\sigma_1 \vv_1$ and $\sigma_2 \vv_2$
would coincide with the principal axes of the ellipse $\cE(\vA)$. An
example of the visualization is presented in \cref{fig:viz} for the
matrices
\begin{equation}
\label{eq:a1a2}
 \vA_1 = \begin{bmatrix} 1&-1\\ 0&1 \end{bmatrix} 
 ~\text{and}~
 \vA_2 = \begin{bmatrix} 1&2\\ 2&1 \end{bmatrix}.
\end{equation}
Finally we note that the ellipses we display in all the following
figures are normalized to an appropriate size, and that when the matrix
$\vA$ is complex, we simply visualize its real and imaginary part
separately, using the same convention.
\begin{figure}
\begin{center}
  \includegraphics[width=0.2\textwidth]{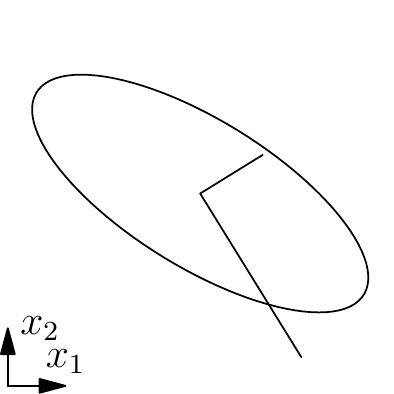}
  \hspace{0.1\textwidth}
  \includegraphics[width=0.2\textwidth]{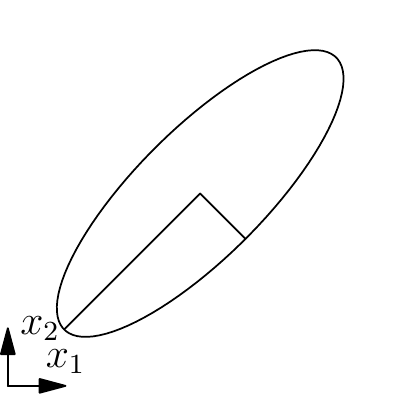}
 \end{center}
 \caption{Sample visualization using ellipses for a non-symmetric matrix
 $\vA_1$ (left) and a symmetric matrix $\vA_2$ (right), as defined in
 \eqref{eq:a1a2}. The ellipses
 have been normalized so that the longest principal axes have the same length. The
 departure of the vectors from the principal axes indicates a
 non-symmetric matrix.}
 \label{fig:viz}
\end{figure}

\subsection{Deterministic experiments}\label{sec.numdet}
For each $\vx_r\in\cA$ and for $501$ uniformly spaced frequencies in the
frequency band $\cB/(2\pi)=[1.2,3.6]$ GHz, we generate coherency matrix
data $\vPsi(\vx_r;\omega)$ using formula \eqref{eq:cohmat2} with a
constant source coherency matrix 
\[
\widetilde{\vJ}_s(\omega) \equiv \begin{bmatrix}1 & 0\\0 & 1\end{bmatrix}.
\]
Using \eqref{eq:prepro}, we preprocess the data for each $\vx_r\in\cA$
and $\omega\in\cB$ to recover the array response tensor $\vp(\vPsi)
(\vx_r,\vx_s;k)$, and then compute the Kirchhoff imaging function
$\cI_{KM}[ \vp(\vPsi) ](\vy;\omega/c)$ for each $\vy$ in the
imaging window $\cW$. Reconstruction of the projected polarizability tensor
$\widetilde{\bGa}(\vy)$ is performed by integrating
$\widetilde{\bGa}(\vy;\omega)$ over the frequency band $\cB$ (as in
\eqref{eq.estimatemf}), where we recover $\widetilde{\bGa}(\vy;\omega)$
using a slightly different formula (see \cref{rem:numalpharec}):
\begin{equation}\label{eq.reconstructsingle2}
\widetilde{\bGa}(\vy;\omega) =
\big( \vU_{\parallel}^{*} \vH_{r}(\vy,\vy;k) \vU_{\parallel}\big)^{-1} \widetilde{\cI}_{KM}
(\vy;\omega/c)\big(  \vU_{s}^{*} \vH_{s}
(\vy,\vy;k)  \vU_{s}\big)^{-1}.
\end{equation}
Finally, we perform the phase correction described in 
\cref{sec.phasecorrection}.

For our first experiment, we consider three dipole scatterers located at
${\vy}_1 = (-6\lambda_0,-5\lambda_0,100\lambda_0)$, ${\vy}_2 =
(7\lambda_0,-5\lambda_0,100\lambda_0)$, and ${\vy}_3 =
(5\lambda_0,8\lambda_0,106\lambda_0)$, with respective polarizability
tensors
\begin{equation}\label{eq:num:poltens}
\begin{array}{c}
\bGa_1 = 
\begin{bmatrix}
2+\imath& -\imath& 1\\
-\imath & 1+2\imath& \imath\\
1 & \imath & 1+\imath
\end{bmatrix},\quad
\bGa_2 = 
\begin{bmatrix}
2+2\imath & -1+\imath & \imath/2\\
-1+\imath & 1+2\imath & 0\\
\imath/2 & 0 & 1
\end{bmatrix},\\
\bGa_3 = 
\begin{bmatrix}
2-2\imath & 1+\imath & 0\\
1+\imath & 1+2\imath & (1-\imath)/2\\
0 & (1-\imath)/2  & \imath
\end{bmatrix}.
\end{array}
\end{equation}
\Cref{fig:cr:det} compares cross-range images formed from the true array
response $\vPi$ and the recovered array response 
$\vp(\vPsi)$, at the range locations $x_3=100\lambda_0$ and
$x_3=106\lambda_0$. The colormap indicates the Frobenius norm of the
recovered polarizability tensor
$\|\widetilde{\bGa}(\vy)\|_F$ at each image point $\vy$. The true 
polarizability tensor norms are given by
$\|\widetilde{\bGa}_1\|_F\approx3.44,
\|\widetilde{\bGa}_2\|_F\approx3.93$,
$\|\widetilde{\bGa}_3\|_F\approx3.82$, thus we observe accurate
recovery of the norm. Meanwhile, the ellipses/axes depict the $2\times
2$ tensors $\widetilde{\bGa}_i$ at the exact scatterer locations
$\vy_i$. We use solid white (resp.\@ dashed black) and solid yellow
(resp.\@ dashed magenta) ellipses/axes for the real and imaginary parts
of the true (resp.\@ recovered) tensor $\widetilde{\bGa}_i$, using
the visualization convention in \cref{sec.viz}. Both true and recovered
tensor are displayed with the phase correction in
\cref{sec.phasecorrection}. In a
similar fashion, \cref{fig:r:det} compares range images at the
cross-range locations $x_2=-5\lambda_0$ and $x_2=8\lambda_0$. In both
figures, the images are essentially indistinguishable. 

\begin{figure}
\begin{center}
\includegraphics[width=0.8\textwidth]{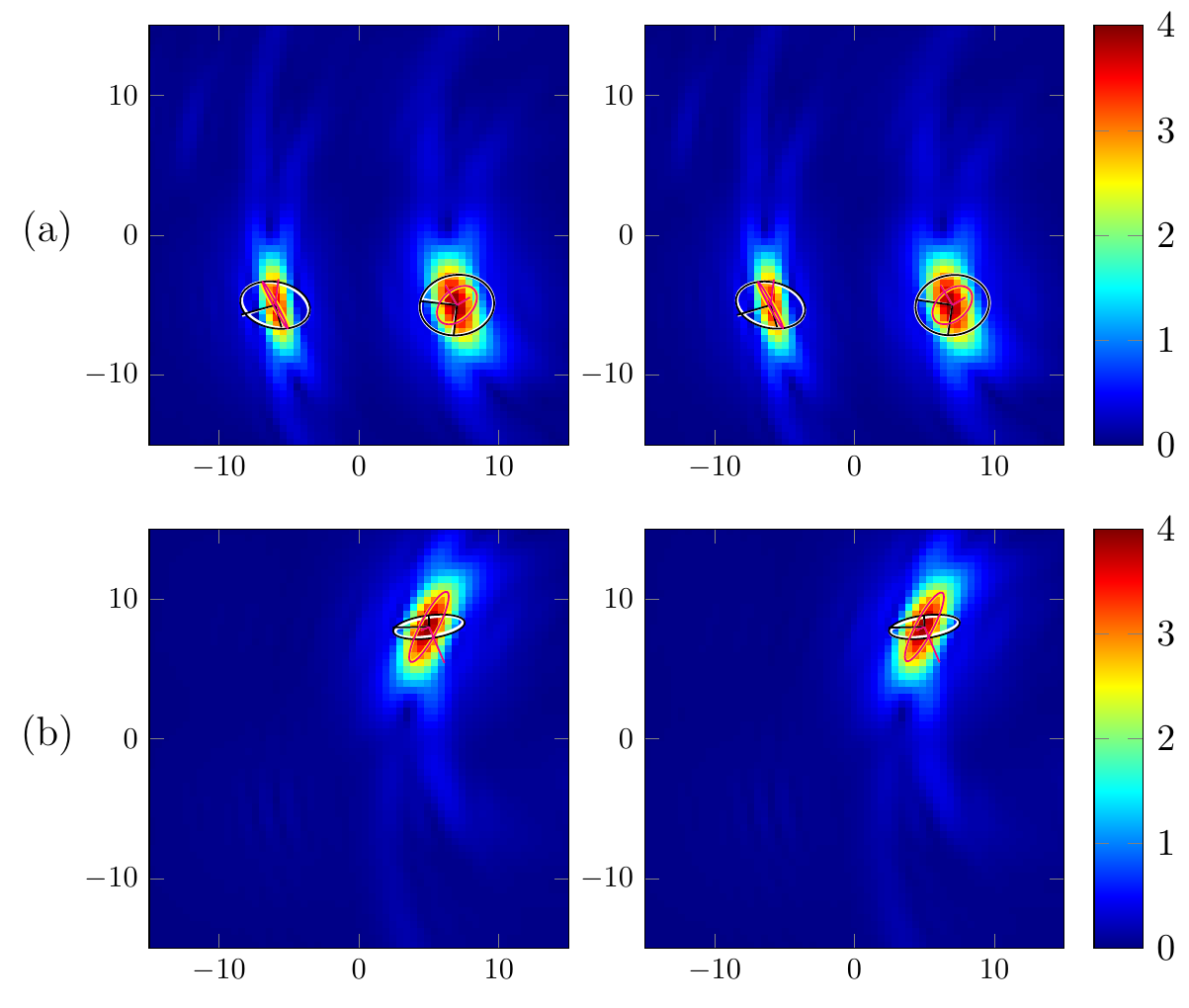}
\end{center}
\caption{Dipole scatterers: cross-range images of
$\|\widetilde{\bGa}(\vy)\|_F$ at range locations (a) $x_3=100\lambda_0$ and
(b) $x_3=106\lambda_0$. The left and right columns show reconstructions
from the true array response $\vPi$ and the recovered array response
$\vp(\vPsi)$, respectively. Here the true tensor
$\Re(\widetilde{\alpha}_{1,1}/|\widetilde{\alpha}_{1,1}|\widetilde{\bGa})$
(resp.
$\Im(\widetilde{\alpha}_{1,1}/|\widetilde{\alpha}_{1,1}|\widetilde{\bGa})$)
is depicted by the white (resp.\@ yellow) ellipses/axes. Similarly, the
recovered tensor is depicted using black (real) and magenta (imaginary)
ellipses/axes.} \label{fig:cr:det}
\end{figure}

\begin{figure}
\begin{center}
\includegraphics[width=0.8\textwidth]{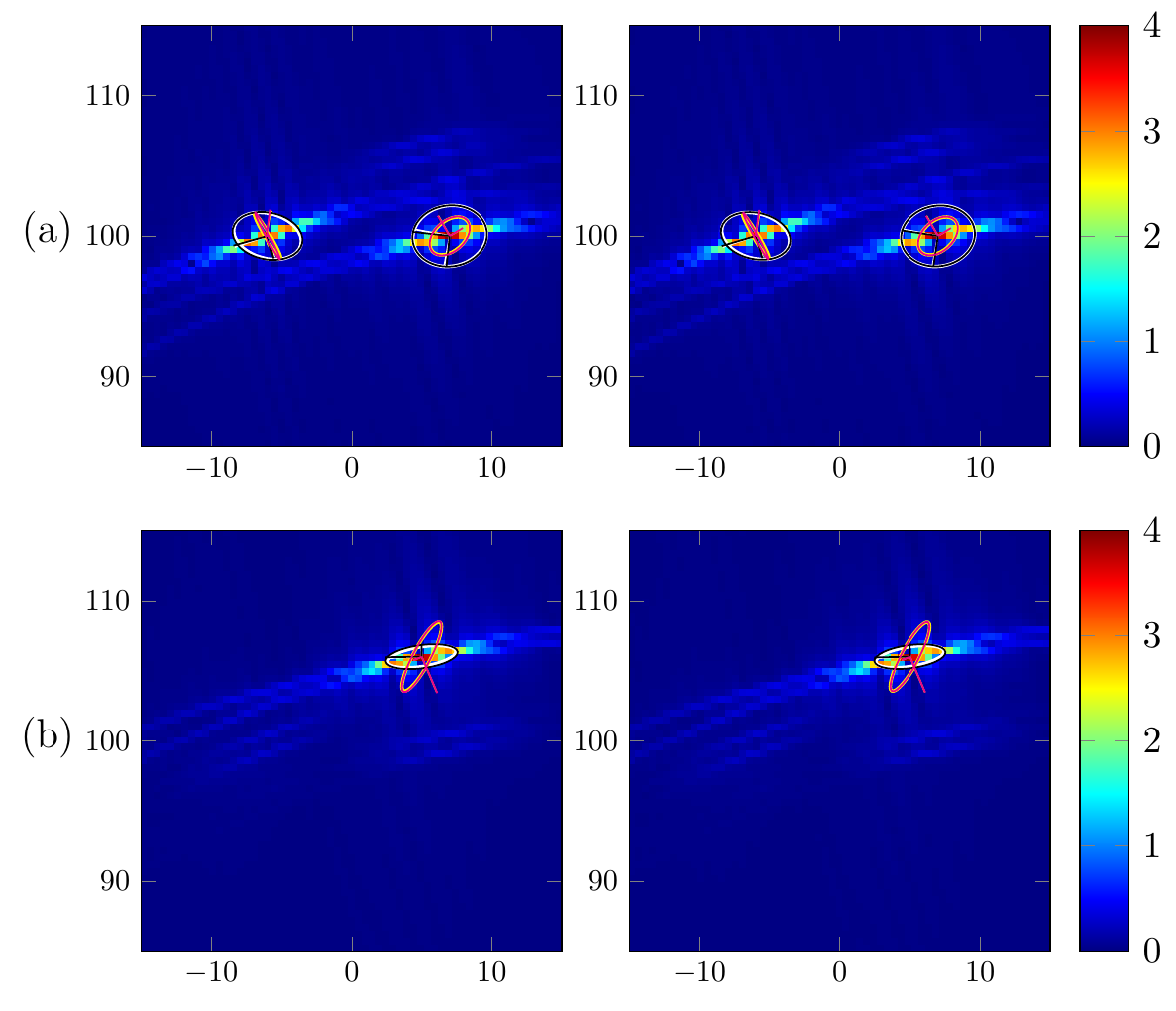}
\end{center}
\caption{Dipole scatterers: range images of
$\|\widetilde{\bGa}(\vy)\|_F$ at cross-range locations (a)
$x_2=-5\lambda_0$ and (b) $x_2=8\lambda_0$.  The left and right columns
show reconstructions from the true array response $\vPi$ and the
recovered array response $\vp(\vPsi)$, respectively. Here the true
tensor
$\Re(\widetilde{\alpha}_{1,1}/|\widetilde{\alpha}_{1,1}|\widetilde{\bGa})$
(resp.
$\Im(\widetilde{\alpha}_{1,1}/|\widetilde{\alpha}_{1,1}|\widetilde{\bGa})$)
is depicted by the white (resp.\@ yellow) ellipses/axes. Similarly, the
recovered tensor is depicted using black (real) and magenta (imaginary)
ellipses/axes.} \label{fig:r:det}
\end{figure}

Next, we demonstrate the oscillatory artifacts in the recovery of the
projected polarizability tensor
$\widetilde{\bGa}(\vy)$ and their correction, as explained in
\cref{sec.phasecorrection}. In \cref{fig:cr:stable}(a), we show a
cross-range image of the recovered tensor field
$\widetilde{\bGa}(\vy)$ near the scatterer location
$\vy_1=(-6\lambda_0,-5\lambda_0,100\lambda_0)$. We use the convention of
\cref{sec.viz} to represent with ellipses the projected polarizability
tensor $\widetilde{\bGa}(\vy)$, but omit the axes representing the
right singular vectors to avoid over-cluttering the image. The
ellipses all have the same longest principal axis and their color
corresponds to $\|\widetilde{\bGa}(\vy)\|_F$. We immediately notice that
the tensor oscillates wildly away from the scatterer position, which
confirms the analysis of \cref{sec.phasecorrection}.  In
\cref{fig:cr:stable}(b), we visualize the tensor field
$\overline{\widetilde{\alpha}_{1,1}(\vy)}/|\widetilde{\alpha}_{1,1}(\vy)|\widetilde{\bGa}(\vy)$
in a similar fashion. This imposes that $\widetilde{\alpha}_{1,1}(\vy)$
be real. The images in range for the same scatterer located at $\vy_1$
appear in \cref{fig:r:stable}. Oscillatory artifacts are present away
from the scatterer location and are suppressed using the same method. 
\begin{figure}
\begin{center}
\includegraphics[width=0.8\textwidth]{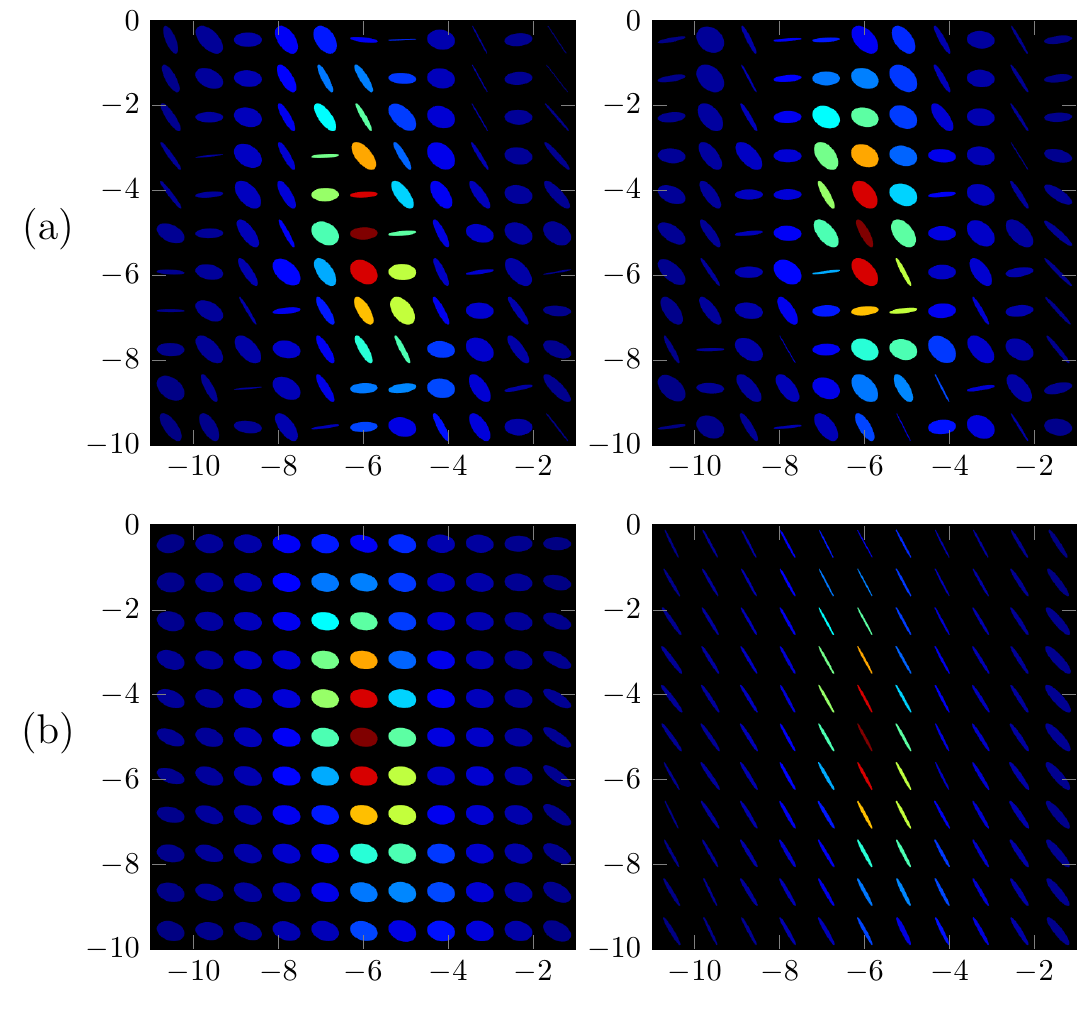}
\end{center}
\caption{Cross-range visualization of the recovered tensor
field $\widetilde{\bGa}(\vy)$ near $
(-6\lambda_0,-5\lambda_0,100\lambda_0)$. Row (a) shows the tensor field
$\widetilde{\bGa}$ with oscillatory artifacts, while row (b) shows the
field
$\overline{\widetilde{\alpha}_{1,1}}/|\widetilde{\alpha}_{1,1}|\widetilde{\bGa}$
with suppressed oscillatory artifacts.
The left and right columns show the real and imaginary parts of the
tensors, respectively.}
\label{fig:cr:stable}
\end{figure}

\begin{figure}
\begin{center}
\includegraphics[width=0.8\textwidth]{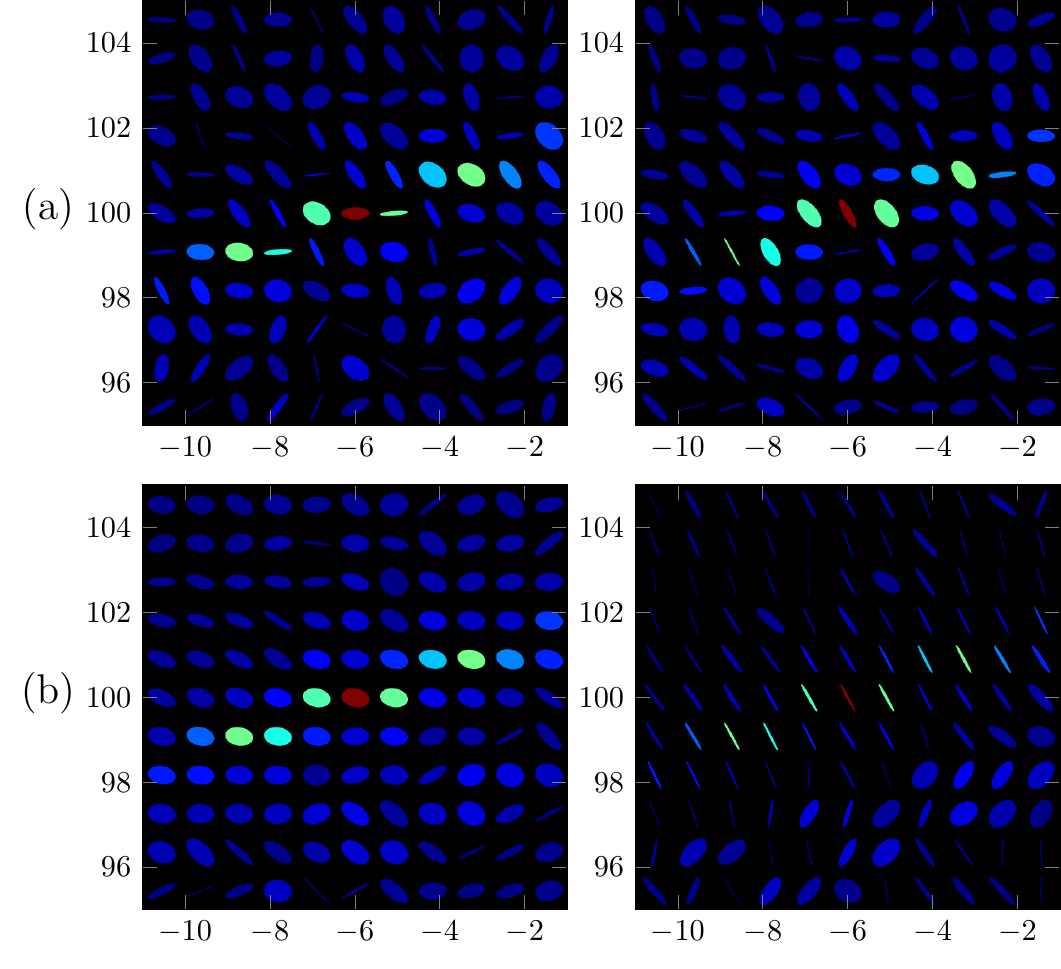}
\end{center}
\caption{Stability: range visualization of the recovered tensor field
$\widetilde{\bGa}(\vy)$ near $
(-6\lambda_0,-5\lambda_0,100\lambda_0)$. Row (a) shows the tensor field
$\widetilde{\bGa}$, while row (b) shows the stabilized tensor field
$\overline{\widetilde{\alpha}_{1,1}}/|\widetilde{\alpha}_{1,1}|\widetilde{\bGa}$.
The left and right columns show the real and imaginary parts of the
tensors, respectively.}
\label{fig:r:stable}
\end{figure}

As a final deterministic experiment, we consider an extended 
scatterer. For simplicity we choose  a cube of side $5\lambda_0$ centered at $(0,0,100\lambda_0)$.
We simulate this scatterer using a collection of dipole scatterers
separated by $\lambda_0/4$, each with polarizability tensor
\[
\bGa_0 = 
\begin{bmatrix}
2-\imath& 3+2\imath& 0\\
3+2\imath& -1& 0\\
0& 0& 1
\end{bmatrix}.
\]
We visualize $\|\widetilde{\bGa}(\vy)\|_F$ in cross-range and range in
\cref{fig:ext}(a) and \cref{fig:ext}(b) respectively, where the outline
of the true scatterer is indicated in cyan. In the range image we find
that
the Kirchhoff imaging routine only resolves discontinuities in the
medium wave speed, while the interior of the cube is not accurately
recovered.  This is similar to what happens in acoustics, see e.g.
\cite{Blei:2013:MSI}. In \cref{fig:ext:stable}, we visualize the
corresponding tensor field with suppressed oscillatory artifacts
$\overline{\widetilde{\alpha}_{1,1}(\vy)}/|\widetilde{\alpha}_{1,1}(\vy)|\widetilde{\bGa}(\vy)$.

\begin{figure}
\begin{center}
\includegraphics[width=0.8\textwidth]{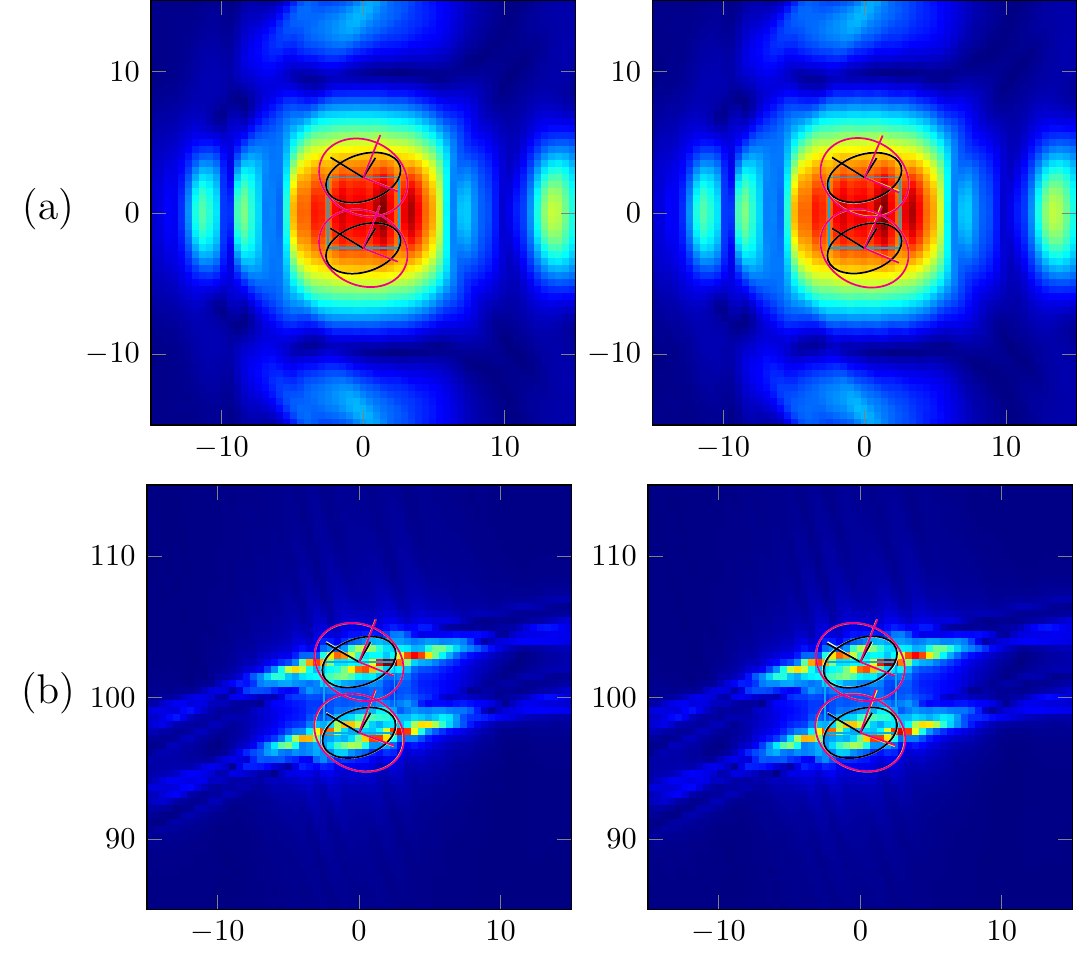}
\end{center}
\caption{Extended scatterer (cube): visualization of
$\|\widetilde{\bGa}(\vy)\|_F$ in (a) cross-range at $x_3=100\lambda_0$ and
(b) range at $x_2=0$. The left and right columns show reconstructions from
the true array response $\vPi$ and the recovered array response
$\vp(\vPsi)$, respectively.}\label{fig:ext}
\end{figure}

\begin{figure}
\begin{center}
\includegraphics[width=0.8\textwidth]{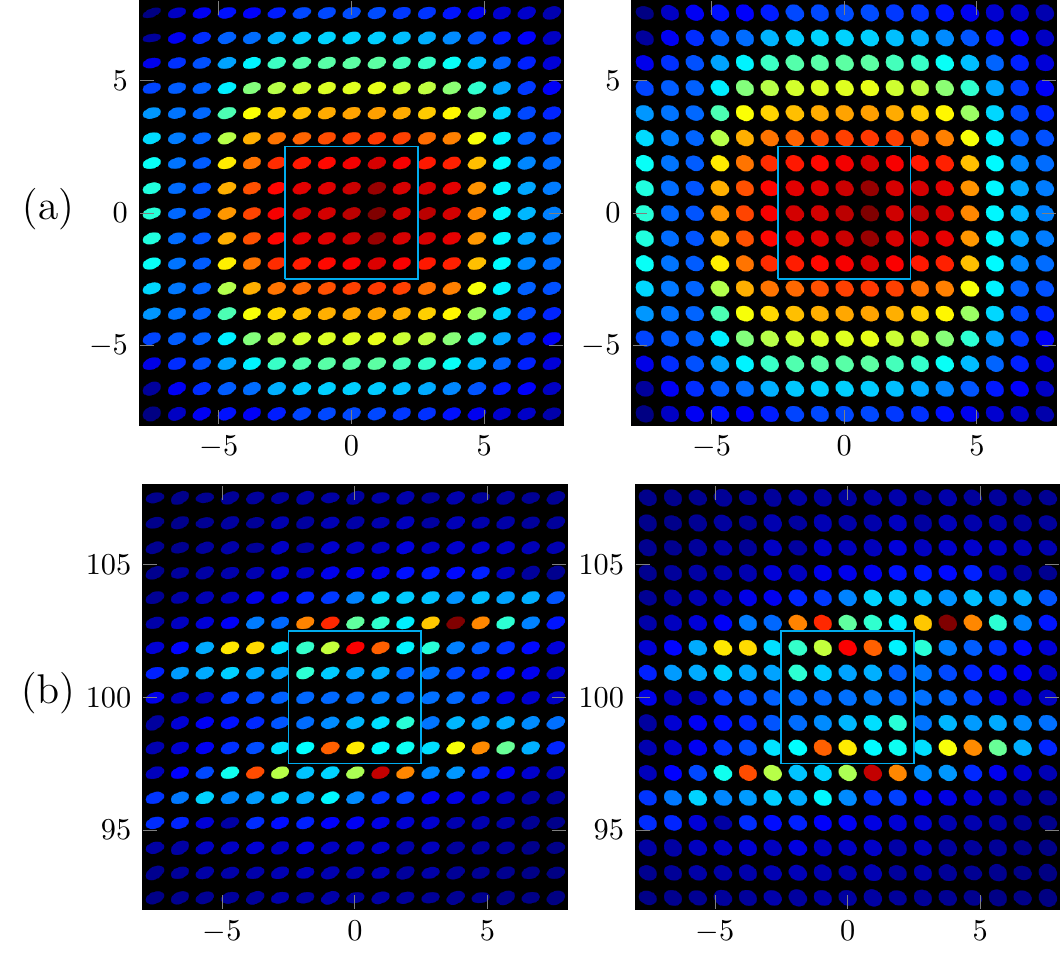}
\end{center}
\caption{Extended scatterer (cube): visualization of the stabilized tensor
field in (a) cross-range at $x_3=100\lambda_0$ and (b) range at $x_2=0$.
The left column depicts
$\Re(\overline{\widetilde{\alpha}_{1,1}}/|\widetilde{\alpha}_{1,1}|\widetilde{\bGa})$,
while the right column depicts
$\Im(\overline{\widetilde{\alpha}_{1,1}}/|\widetilde{\alpha}_{1,1}|\widetilde{\bGa}).$}
\label{fig:ext:stable}
\end{figure}

From the experiments in this section, we confirm that combining 
\cref{thm:geoimg} and \cref{prop.incompletedata}, one obtains a
good reconstruction of $\widetilde{\bGa}$ using only coherency matrix data $\vp(\vPsi)$. On one hand,
\cref{thm:geoimg} tells us that
$\cI_{KM}[\vp(\vPsi)](\vy;\omega/c)$ is asymptotically close to
$\cI_{KM}[\vU_{\parallel} \widetilde{\vPi} \vU_s^{*}](\vy;\omega/c) =
\cI_{KM}[\vP_{\parallel} \vPi \vP_s](\vy;\omega/c)$ if the frequency is
sufficiently high. On the other hand, in the Fraunhofer regime,
\cref{prop.incompletedata} tells us that both images
$\cI_{KM}[\vP_{\parallel} \vPi \vP_s](\vy;\omega/c)$ (with incomplete
polarization data) and $\cI_{KM}[ \vPi ](\vy;\omega/c)$ (with full data)
have the same asymptotic expansion. Thus, by post-processing
$\cI_{KM}[\vP_{\parallel} \vPi \vP_s ](\vy;\omega/c)$ (according to
relations \eqref{eq.estimatemf}, \eqref{eq.polareconst}, and correction
of the oscillatory artifacts explained in \cref{sec.phasecorrection}), one
obtains a good reconstruction of the position and polarizability tensor
$\widetilde{\bGa}$ of each dipole with the same resolution (both in
cross-range and range) as in the full data case. 

\begin{remark}\label{rem:decay}
Note that the decay of the cross-range Kirchhoff images
as we move away
from the scatterers
(\cref{fig:cr:det,fig:r:det,fig:cr:stable,fig:r:stable}) is slower than the decay that was observed
for the collocated sources and receivers array case in
\cite{Cassier:2017:IPD}. The decay of the images in our case is given by
\cref{prop.crossrang,prop.rangeresolution} and is inversely proportional
to the cross-range distance from a scatterer to the imaging point. For
the collocated sources and receiver case, the decay is inversely
proportional to the distance squared.
\end{remark}

\begin{remark}\label{rem:numalpharec}
We use \eqref{eq.reconstructsingle2} instead of \eqref{eq.polareconst} to
recover $\widetilde{\bGa}(\vy;\omega)$ for each $\omega\in\cB$. However,
the reconstruction formulas \eqref{eq.reconstructsingle2} and
\eqref{eq.polareconst} are asymptotically close since
\eqref{eq.polareconst} amounts to using the Fraunhofer asymptotic of
$\vU_{\parallel}^{*} \vH_{r}(\vy,\vy;k) \vU_{\parallel}$ and
$\vU_{s}^{*} \vH_{s}(\vy,\vy;k) \vU_{s}$ to reconstruct
$\widetilde{\bGa}$. Indeed, thanks to
\eqref{eq.spreadfunctionarray} and \eqref{eq.asympspreadfunctionsource},
one can show
\begin{eqnarray*} \vU_{\parallel}^{*} \vH_{r}(\vy,\vy;k)
\vU_{\parallel}&=&\frac{\operatorname{mes} \cA}{(4\pi L)^2} \vI+
\cO\left(\frac{a^4\Theta_a}{L^4} \right)+
\cO\left(\frac{a^2\Theta_b}{L^2} \right)+\cO\left(\frac{a^3}{L^3}
\right),\\
\vU_{s}^{*} \vH_{s}(\vy,\vy;k) \vU_{s}&=&\frac{1}{(4\pi L)^2}
\vI+O\left(\frac{b}{L}\right)+\cO\left(\frac{a^2}{\Theta_a L^2}\right),
\end{eqnarray*}
where $\vI$ is the $2\times 2$ identity.
Thus, $\widetilde{\bGa}$ reconstructed with
\eqref{eq.reconstructsingle2} satisfies: $$
\widetilde{\bGa}(\vy;\omega)= \frac{(4\pi L)^4}{ \operatorname{mes}\cA}
\, \widetilde{\cI}_{KM}(\vy;k)+ \cO\left(\frac{a^2\Theta_a}{L^2} \right)+
\cO\left(\Theta_b \right)+\cO\left(\frac{a}{L} \right),
$$
where the leading order term corresponds to \eqref{eq.polareconst}. We
use \eqref{eq.reconstructsingle2} instead of \eqref{eq.polareconst}
as it may be more robust in other settings than the Fraunhofer asymptotic regime.
\end{remark}

\subsection{Stochastic experiment} 
\label{sec.numstoc}
Here we perform an experiment where
we drive the electric dipole source at $\vx_s$ using a stochastic polarization
vector $\vj_s(t)$. For this experiment, we use the
same 3-dipole scatterer setup given in \cref{sec.numdet}.

Making use of a vector version of the Wiener-Khinchin theorem, we 
generate $\vj_s(t)$ as a real Gaussian process satisfying
\begin{equation}\label{eq:stocsrcnum1}
\langle \vj_s(t)\rangle = \vzero~\text{and}~
\langle \vj_s(t+\tau)\vj_s(t)^T\rangle = \vU_s J(\tau)\begin{bmatrix} 1
& 0 \\ 0 & 1\end{bmatrix}\vU_s^*,
\end{equation}
where
\begin{equation}\label{eq:stocsrcnum2}
J(\tau) = 4\pi t_c^{-1}\cos(\omega_0
\tau)\exp[-\pi(\tau/t_c)^2].
\end{equation}
Here $t_c$ denotes the correlation time, which we set to $t_c \approx 1$ ns
giving the signal a frequency band of roughly $\cB/(2\pi) = [1.2,3.6]$
GHz. We generate signals of length $2T$ for $T\approx 266$ ns with 8001
uniformly spaced samples. This sampling is sufficient to resolve
frequencies in $\cB$, while $T$ is long enough to observe good ergodic
averaging (see \cref{sec:stochillum}).

We generate the 
electric field corresponding to $\vj_s(t)$, given in the frequency
domain by the first Born approximation as
\[
\vE(\vx_r,\omega) =
\left(\vG(\vx_r,\vx_s;\omega/c)+\vPi(\vx_r,\vx_s;\omega/c)\right)\vj_s(\omega).
\]
To avoid problems with circular correlation in numerics, we 
transform this signal to the time domain, appropriately zero-pad the
result  and then transform the padded time signal back to the frequency
domain (see e.g., \cite[Appendix B]{Garnier:2013:SNR}). Finally, the
empirical correlations in \eqref{eq:empauto} are generated for each
$\vx_r\in\cA$ as 
\[
\vpsi_{emp}(\vx_r,\tau) = \frac{2\pi}{2T}\int d\omega
e^{-i\omega\tau}\vU_\parallel^*\vE(\vx_r,\omega)\overline{\vE(\vx_r,\omega)}^T\vU_\parallel.
\]
We preprocess the empirical correlations according to \eqref{eq:prepro}
to obtain $\vp(\vpsi_{emp})$, and then form the Kirchhoff image
functions $\cI_{KM}[\vp(\vpsi_{emp})](\vy;k)$. The recovery of the
projected polarizability tensor $\widetilde{\bGa}(\vy)$ is performed as
in the deterministic setting (i.e., using \eqref{eq.reconstructsingle2}
to recover $\widetilde{\bGa}(\vy;\omega)$ followed by integration over
the frequency band as in \eqref{eq.estimatemf}).

In \cref{fig:cr:stoc}, we show the cross-range images formed using the
true array response $\vPi$, and the recovered array response
$\vp(\vpsi_{emp})$. Here we see, just as in the deterministic
setting, good recovery of both the locations and tensors of each dipole
scatterer $\widetilde{\bGa}(\vy_i)$. In \cref{fig:stoc:stable}, we again
demonstrate our proposed phase correction method and the accurate tensor
recovery it provides.

\begin{figure}
\begin{center}
\includegraphics[width=0.8\textwidth]{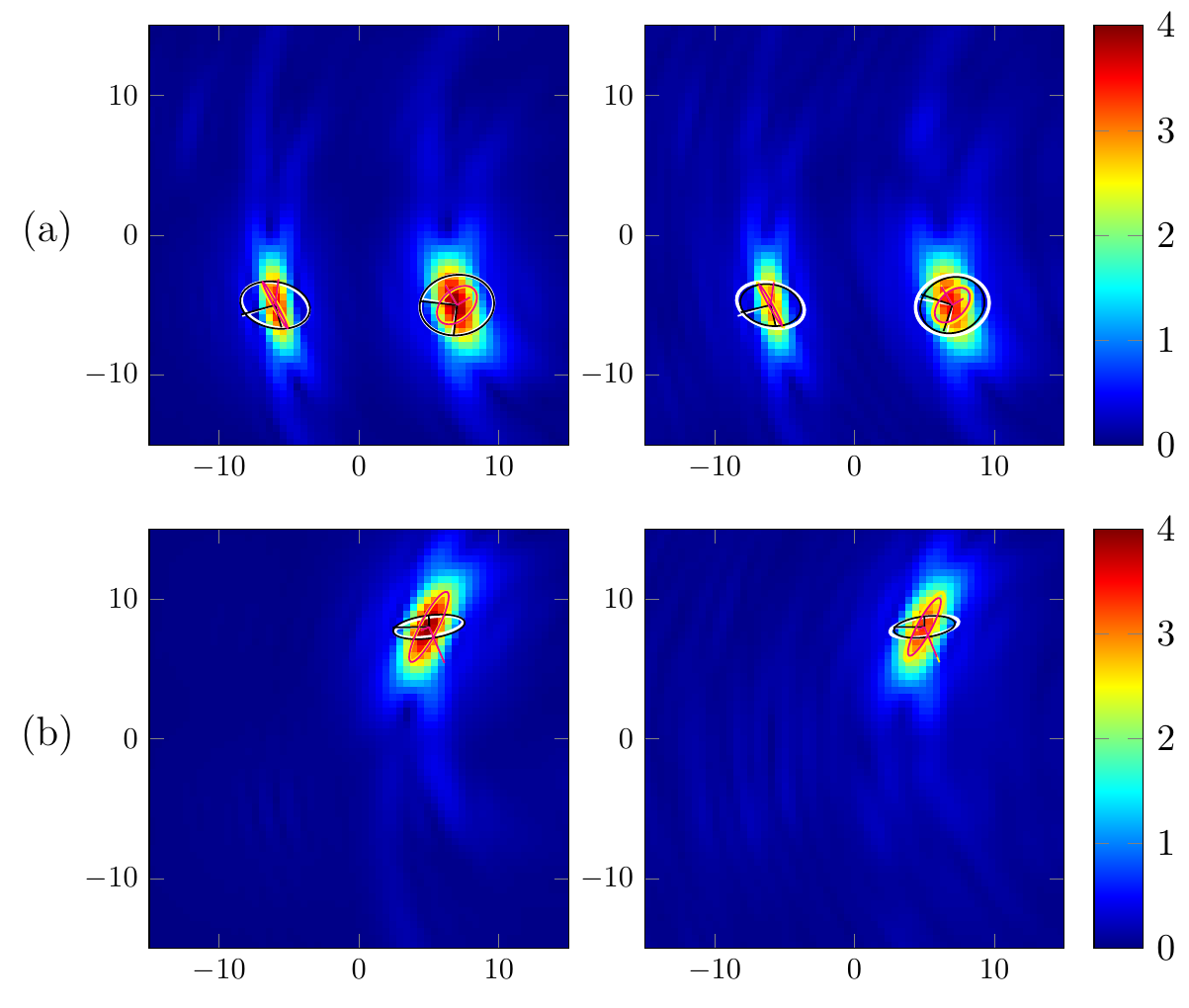}
\end{center}
\caption{Stochastic experiment: cross-range images of
$\|\widetilde{\bGa}(\vy)\|_F$ at range locations (a) $x_3=100\lambda_0$
and (b) $x_3=106\lambda_0$. The left and right columns show
reconstructions from the true array response $\vPi$ and the recovered
array response $\widetilde{\vPi}$, respectively. Here the true tensor
$\Re(\widetilde{\alpha}_{1,1}/|\widetilde{\alpha}_{1,1}|\widetilde{\bGa})$
(resp.
$\Im(\widetilde{\alpha}_{1,1}/|\widetilde{\alpha}_{1,1}|\widetilde{\bGa})$)
is depicted by the white (resp.\@ yellow) ellipses/axes.  Similarly, the
recovered tensor is
depicted using black (real) and magenta (imaginary) ellipses/axes.} 
\label{fig:cr:stoc}
\end{figure}

\begin{figure}
\begin{center}
\includegraphics[width=0.8\textwidth]{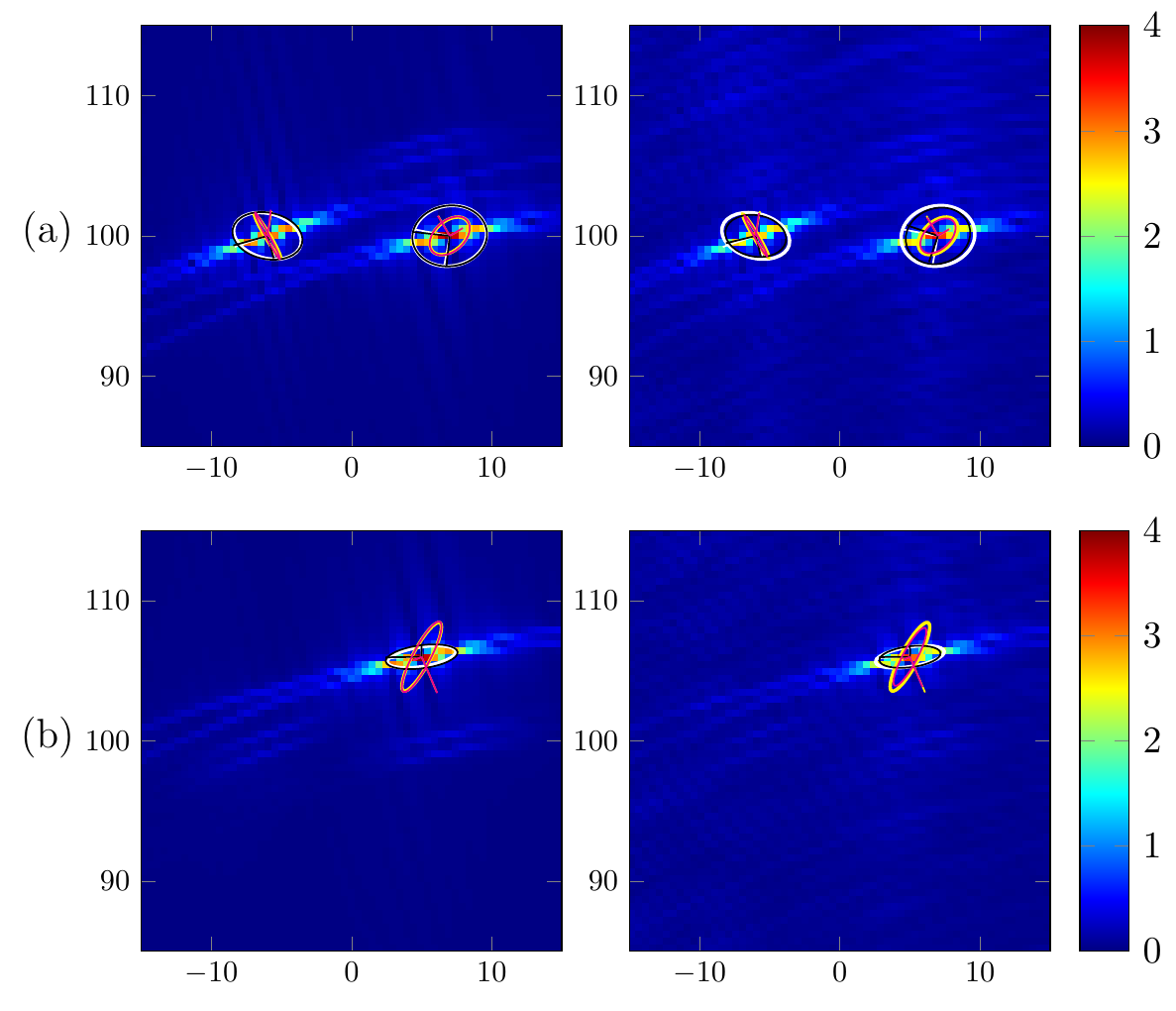}
\end{center}
\caption{Stochastic experiment: range images of
$\|\widetilde{\bGa}(\vy)\|_F$ at cross-range locations (a)
$x_2=-5\lambda_0$ and (b) $x_2=8\lambda_0$.  The left and right columns
show reconstructions from the true array response $\vPi$ and the
recovered array response $\widetilde{\vPi}$, respectively. Here the true
tensor
$\Re(\widetilde{\alpha}_{1,1}/|\widetilde{\alpha}_{1,1}|\widetilde{\bGa})$
(resp.
$\Im(\widetilde{\alpha}_{1,1}/|\widetilde{\alpha}_{1,1}|\widetilde{\bGa})$)
is depicted by the white (resp.\@ yellow) ellipses/axes.  Similarly, the
recovered tensor is depicted using black (real) and magenta (imaginary)
ellipses/axes.}
\label{fig:r:stoc}
\end{figure}

\begin{figure}
\begin{center}
\includegraphics[width=0.8\textwidth]{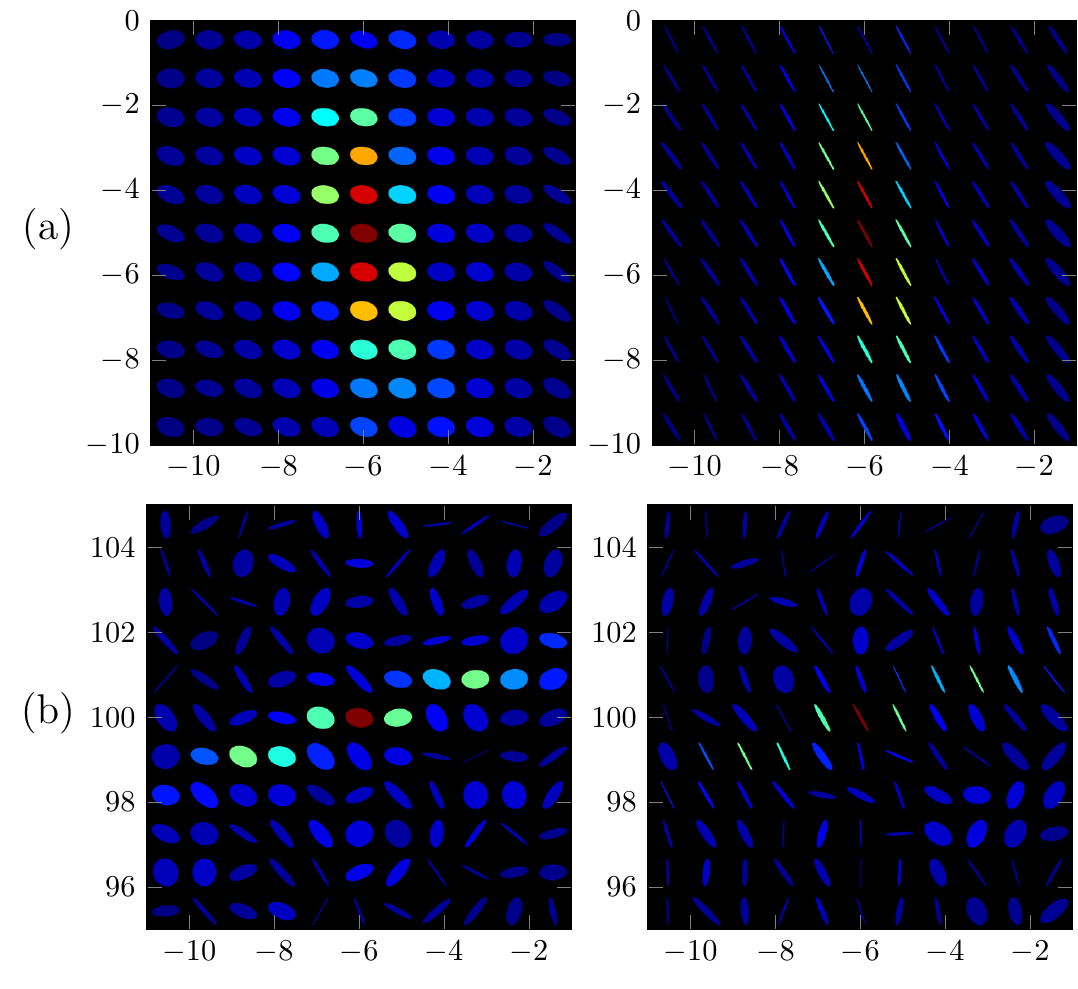}
\end{center}
\caption{Stochastic experiment: visualization of the tensor field with
suppressed oscillatory artifacts in (a) cross-range and (b) range, near
the location $\vy=(-6\lambda_0,-5\lambda_0,100\lambda_0)$. The left
column depicts
$\Re(\overline{\widetilde{\alpha}_{1,1}}|\widetilde{\alpha}_{1,1}|\widetilde{\bGa})$,
while the right column depicts
$\Im(\overline{\widetilde{\alpha}_{1,1}}|\widetilde{\alpha}_{1,1}|\widetilde{\bGa}).$}
\label{fig:stoc:stable}
\end{figure}

\section{Summary and future work}
\label{sec:discussion}
We have introduced a method for imaging the polarization tensor of small
dielectric scatterers in a homogeneous medium, from measurements of the
coherency matrix at an array arising from illumination of a point
source. The main idea in our method is to partially recover the total
electric field at the array and then use it to image using an
electromagnetic version of Kirchhoff migration. We prove using a
stationary phase argument that the error we make in estimating the
electric field does not affect the images. Moreover the images we obtain
are tensor fields that contain information about the components of the
polarization tensor in a certain basis.

There are several ways in which we plan to extend this work. The first
one is to study this problem when the particles are close to an
interface. This is useful in microscopy, where the objects that one
wishes to image are on a substrate and would involve using half-space Green
functions. In the same direction, stratified media Green functions could
be considered. One aspect of our method is that we are
using what are essentially interference patterns between the incident
and scattered fields. In this respect, our method is a form of
holography. It is natural to ask whether we can use other incident
fields to image (e.g. a plane wave). Finally we point out that our
analysis relies on a linearization of the problem (Born approximation).
In reality nearby scatterers could interact. We would like to study this
case using discrete Lippmann-Schwinger models (e.g. the Foldy-Lax model \cite{Malet:2005:MGM,Cassier:2013:MSA}) to find the
locations of scatterers and also correct for any artifacts the multiple
reflections may have introduced.

\section*{Acknowledgements}
This work was partially supported by the National Science Foundation
grant DMS-1411577. The work of MC was supported by Simons Foundation
grant \#376319 (Michael I. Weinstein).  FGV thanks the Laboratoire Jean
Kuntzmann in Grenoble for hosting him while this article was written.
FGV also thanks support from the National Science Foundation grant
DMS-1439786, while FGV was in residence at the Institute for
Computational and Experimental Research in Mathematics in Providence,
RI, during the Fall 2017 semester.

\appendix
\section{Proof of \cref{prop:statstab}}\label{sec:app_statstab} Here we
prove \cref{prop:statstab}. This proposition and its proof are patterned
after a result by Garnier and Papanicolaou \cite[Proposition
4.1]{Garnier:2009:PSI} which was shown in the case of acoustic waves. We
make the necessary modifications to adapt the result to the
electromagnetic setting. Throughout this section we slightly abuse 
notation by identifying functions in the time domain and in the
frequency domain using the same symbol.

\begin{proof}
Let us introduce the notation
\[
 \cvG(\vx_r,\omega) \equiv  \vG(\vx_r,\vx_s;\omega/c) +
 \vPi(\vx_r,\vx_s;\omega/c)
\]
so that the total electric field is $\vE(\vx_r,\omega) =
\cvG(\vx_r,\omega) \vj_s(\omega)$ in the frequency domain. In the time
domain, the total field is 
\[
 \vE(\vx_r,t) = (2\pi)^{-1} [ \cvG(\vx_r,\cdot) \ast \vj_s(\cdot) ](t),
\]
and is a stationary random process because we assume the process
$\vj_s(t)$ driving the source is stationary. Hence we have
\[
\begin{aligned}
\Ma{\vpsi_{emp}(\vx_r,\tau)} &= \frac{1}{2T}\int_{-T}^Tdt
\vU_\parallel^* \Ma{\vE(\vx_r,t+\tau)\vE(\vx_r,t)^T} \vU_\parallel\\
&= \frac{1}{2T}\int_{-T}^T dt \vU_\parallel^* \Ma{\vE(\vx_r,\tau)\vE(\vx_r,0)^T}  \vU_\parallel\\
&= \vU_\parallel^* \Ma{\vE(\vx_r,\tau)\vE(\vx_r,0)^T} \vU_\parallel,
\end{aligned}
\]
and thus $\Ma{\vpsi_{emp}(\vx_r,\tau)}$ is independent of $T$.
Furthermore, we have
\[
\begin{aligned}
\Ma{\vE(\vx_r,\tau)\vE(\vx_r,0)^T} &= \int dt'\int dt''
\cvG(\vx_r,\tau-t')\Ma{\vj_s(t')\vj_s(t'')^T} 
\cvG(\vx_r,-t'')^T
\\
&=\int dt'\int dt''
\cvG(\vx_r,\tau-t')\vU_s \widetilde{\vJ}_s(t'-t'') \vU_s^* \cvG(\vx_r,-t'')^T\\
&=\int
dt''\left[\cvG(\vx_r,\cdot)\ast\vU_s \widetilde{\vJ}_s\vU_s^*\right] 
(\tau-t'')\cvG(\vx_r,-t'')^T\\
&= (2\pi) \int d\omega
e^{-\i\omega\tau}\cvG(\vx_r,\omega)\vU_s \widetilde{\vJ}_s(\omega) \vU_s^*
\cvG(\vx_r,\omega)^*.
\end{aligned}
\]
\Cref{eq:avgempauto} is verified by multiplying the previous expression
on the left by $\vU_\parallel^*$ and on the right by $\vU_\parallel$.

We show the ergodicity in \eqref{eq:ergodic}, by proving that the variance of each component of $\vpsi_{emp}(\vx_r,\tau)$ is $\cO(1/T)$ as $T \to
\infty$. To simplify the expressions we use the notations
$[(\vpsi_{emp})(\vx_r,t)]_{ij} =
\psi_{ij}(t)$, $[\cvG(\vx_r,t)]_{mn}=\cG_{mn}(t)$, $[\vJ_s(t)]_{mn} =
J_{mn}(t)$ and $[\vj_s(t)]_m = j_m(t)$ for $i,j\in\{1,2\}$ and $m,n \in
\{1,2,3\}$.  We compute the covariance of
$\psi_{ij}(\tau)$ with the Einstein summation convention,
\begin{equation}\label{eq:corrcov}
\begin{aligned}
\Cov&\big(\psi_{ij}(\tau),\psi_{i'j'}(\tau+\Delta\tau)\big) =
\frac{1}{(2T)^2}\int_{-T}^T\int_{-T}^Tdt dt'\int dsds'dudu'\\
&\times
\cG_{im}(s+\tau)\cG_{jn}(u)\cG_{i'm'}(s'+\tau+\Delta\tau) 
\cG_{j'n'}(u')\\
&\times\Big[\Ma{j_m(t-s)j_n(t-u)j_{m'}(t'-s')j_{n'}(t'-u')}\\
&-\Ma{j_m(t-s)j_n(t-u)}\Ma{j_{m'}(t'-s')j_{n'}(t'-u')}\Big].
\end{aligned}
\end{equation}
For the following, it is helpful to notice the symmetry $J_{ij}(\tau) = J_{ji}(-\tau)$ in
the time domain, which follows from \eqref{eq:js:t}.
The product of the second order moments is 
\[
\Ma{j_m(t-s)j_n(t-u)}\Ma{j_{m'}(t'-s')j_{n'}(t'-u')}= 
J_{mn}(u-s)J_{n'm'}(s'-u'),
\]
while the fourth order moment is given by the Gaussian moment theorem
\[
\begin{aligned}
\Ma{j_m(t-s)j_n(t-u)j_{m'}(t'-s')j_{n'}(t'-u')}=& 
J_{mn}(u-s)J_{n'm'}(s'-u')\\
+&J_{mm'}(t-s-t'+s')J_{n'n}(t'-u'-t+u)\\
+&J_{mn'}(t-s-t'+u')J_{m'n}(t'-s'-t+u).
\end{aligned}
\]
We evaluate the following integrals as
\[
 \begin{aligned}
  I_1 &= \frac{1}{(2T)^2} \int_{-T}^T dt \int_{-T}^T dt'
  J_{mm'}(t-t'+s'-s) J_{n'n}(t'-t + u - u')\\
  &=\int d\omega \int d\omega' \sinc^2((\omega-\omega')T) e^{-i\omega(s'-s)}
  e^{-i\omega'(u-u')} J_{mm'}(\omega) J_{n'n}(\omega'),
 \end{aligned}
\]
and
\[
 \begin{aligned}
  I_2 &= \frac{1}{(2T)^2} \int_{-T}^T dt \int_{-T}^T dt'
  J_{mn'}(t-t'+u'-s) J_{m'n}(t'-t + u - s')\\
  &=\int d\omega \int d\omega' \sinc^2((\omega-\omega')T) e^{-i\omega(u'-s)}
  e^{-i\omega'(u-s')} J_{mn'}(\omega) J_{m'n}(\omega').
 \end{aligned}
\]
To see where $I_1$ and $I_2$ come from, consider that for
functions $f,g$ and scalars $a,b$ we have
\[
\begin{aligned}
 \frac{1}{(2T)^2}&\int_{-T}^T dt \int_{-T}^T dt' f(t-t'+a) g(t'-t+b)\\
 &=\frac{1}{(2T)^2}\int_{-T}^T dt \int_{-T}^T dt' \int d\omega
 e^{-i\omega(t-t'+a)} f(\omega) \int d\omega'
 e^{-i\omega'(t'-t+b)} g(\omega')\\
 &=\frac{1}{(2T)^2} \int d\omega \int d\omega' f(\omega) g(\omega')
 e^{-i\omega a} e^{-i\omega' b} \int_{-T}^T dt e^{it(\omega'-\omega)}
 \int_{-T}^T dt' e^{it'(\omega-\omega')}\\
 &=\int d\omega \int d\omega' f(\omega) g(\omega')
 e^{-i\omega a} e^{-i\omega' b}
 \left(\frac{\sin(T(\omega'-\omega)}{T(\omega'-\omega)}\right)^2.
\end{aligned}
\]
We split the covariance \eqref{eq:corrcov} into two terms
$\Cov(\psi_{ij}(\tau),\psi_{i'j'}(\tau+\Delta\tau)) =  V_1 + V_2$, where
$V_l$ involves $I_l$, $l=1,2$. Using the expression for $I_1$ in the
first term $V_1$, we can evaluate the integrals in $s,s',u,u'$ to get
\[
\begin{aligned}
 V_1 &=\int ds ds' du du' \cG_{im}(s+\tau)\cG_{jn}(u)\cG_{i'm'}(s'+\tau+\Delta\tau) 
\cG_{j'n'}(u') I_1\\
 &= (2\pi)^4 \int d\omega \int d\omega' e^{i\omega\Delta\tau} \sinc^2((\omega-\omega')T) 
 \cG_{im}(\omega)   \cG_{i'm'}(-\omega) 
 \cG_{jn}(-\omega') \cG_{j'n'}(\omega')\\
 &\quad\quad\quad\quad\times
 J_{mm'}(\omega) J_{n'n}(\omega').
\end{aligned}
\]
The details for the calculation of $V_1$ are as follows,
\[
\begin{aligned}
\int ds e^{i\omega s} \cG_{im}(s+\tau) &=(2\pi) e^{-i\omega \tau}
\cG_{im}(\omega),\\
\int ds' e^{-i\omega s'} \cG_{i'm'}(s'+\tau + \Delta \tau) &=(2\pi)
e^{i\omega(\tau+\Delta\tau)}\cG_{i'm'}(-\omega),\\
\int du e^{-i\omega'u} \cG_{jn}(u) &= (2\pi)
\cG_{jn}(-\omega'),~\text{and}\\
\int du' e^{i\omega'u'} \cG_{j'n'}(u') &= (2\pi) \cG_{j'n'}(\omega').
\end{aligned}
\]
Similarly for the second term $V_2$ we obtain
\[
\begin{aligned}
 V_2 &= (2\pi)^4 \int d\omega \int d\omega'
 e^{-i\omega\tau}e^{-i\omega'(\tau+\Delta\tau)}
 \sinc^2((\omega-\omega')T)\\
 &\quad\quad\quad\quad\quad\times \cG_{im}(\omega)   \cG_{i'm'}(\omega') 
 \cG_{jn}(-\omega') \cG_{j'n'}(-\omega)
 J_{mn'}(\omega) J_{m'n}(\omega').
\end{aligned}
\]
The details for the calculation of $V_2$ are as follows,
\[
\begin{aligned}
\int ds e^{i\omega s} \cG_{im}(s+\tau) &=(2\pi) e^{-i\omega \tau}
\cG_{im}(\omega),\\
\int ds' e^{i\omega' s'} \cG_{i'm'}(s'+\tau + \Delta \tau) &=(2\pi)
e^{-i\omega'(\tau+\Delta\tau)}\cG_{i'm'}(\omega'),\\
\int du e^{-i\omega'u} \cG_{jn}(u) &= (2\pi)
\cG_{jn}(-\omega'),~\text{and}\\
\int du' e^{-i\omega u'} \cG_{j'n'}(u') &= (2\pi) \cG_{j'n'}(-\omega).
\end{aligned}
\]
Since we have
\[
 T \int \sinc^2(T\omega) d\omega = \pi,
\]
we see that when we take $\Delta \tau = 0$ and the limit as $T \to \infty$ we get
\[
\begin{aligned}
 \frac{T}{\pi} V_1 \to & L_1 \equiv (2\pi)^4 \int d\omega \cG_{im}(\omega)
 \cG_{i'm'}(-\omega) \cG_{jn}(-\omega) \cG_{j'n'}(\omega) J_{mm'}(\omega)
 J_{n'n}(\omega),~\text{and}\\
 \frac{T}{\pi} V_2 \to & L_2 \equiv (2\pi)^4 \int d\omega e^{-2i\omega\tau}
 \cG_{im}(\omega)  \cG_{i'm'}(\omega) \cG_{jn}(-\omega)
 \cG_{j'n'}(-\omega)J_{mn'}(\omega) J_{m'n}(\omega),
\end{aligned}
\]
using an approximate Dirac identity, which we can use e.g. when
$J_{ij}(\omega)$ is Schwartz class.
Notice that the limiting values $L_1$ and $L_2$ are guaranteed to be real because they
involve the Fourier transform of functions of $\omega$ that satisfy
an appropriate reflection principle. Evaluating
$\Cov(\psi_{ij}(t),\psi_{i'j'}(t))$ on the diagonal (i.e., $i=i',j=j'$),
we conclude that $T\Var(\psi_{ij}(\tau)) = \cO(1)$ as $T\to\infty$,
which establishes \eqref{eq:ergodic}.
\end{proof}	

\section{Conditioning of projected Green function matrix}%
\label{app:gtilde}
Here we show that the  matrix $\vP_\parallel \vG(\vx_r,\vx_s;k) \vP_s$
is close to a rank $2$ matrix for which the condition number can be
calculated explicitly in terms of the angles of the triangle spanned by
$\vx_r$, $\vx_s$ and $\vy_0$. We recall that the condition number of a rank
$r$ matrix $\vA$ is $\cond(\vA) = \sigma_1(\vA) /
\sigma_r(\vA)$, where $\sigma_j(\vA)$ is the $j-$th singular value of
$\vA$. Since $\widetilde{\vG}(\vx_r,\vx_s;k)= \vU_\parallel^*
\vG(\vx_r,\vx_s;k) \vU_s$  this
result implies $\widetilde{\vG}(\vx_r,\vx_s;k)$ is close to a $2\times
2$ invertible matrix, because $\vP_\parallel = \vU_\parallel
\vU_\parallel^*$ and $\vP_s = \vP(\vx_s,\vy_0) = \vU_s \vU_s^*$, with
the $3\times 2$ matrices $\vU_\parallel$ and $\vU_s$ being unitary.
\begin{lemma}
\label{lem:pgp}
Under the Fraunhofer asymptotic regime (see \cref{sec:fraunhofer}) we
have that
\begin{equation}
 \vP_\parallel \vG(\vx_r,\vx_s;k) \vP_s = G(\vx_s,\vx_r;k) \left[ \vP(\vx_r,\vy_0)
 \vP(\vx_s,\vx_r)  \vP_s + \cO\M{\frac{a}{L}} + \cO\M{\frac{1}{kd}}
 \right].
 \label{eq:pgp}
\end{equation}
Note that we neglected $\cO(a/(kdL))$ because 
$kd \gg 1$ and $a/L \ll 1$.
Moreover the condition number of $\vP(\vx_r,\vy_0) \vP(\vx_s,\vx_r)
\vP(\vy_0,\vx_s)$ is $|\cos(\theta_r)\cos(\theta_s)|^{-1}$, where
$\theta_j$ is the angle at the vertex $\vx_j$ of the triangle with
vertices $\vx_r$, $\vx_s$ and $\vy_0$, for $j=r,s$.
\end{lemma}
\begin{proof}
We can use
\cite[eq. (7)]{Cassier:2017:IPD} to see that
\begin{equation}
 \vG(\vx_s,\vx_r;k) = G(\vx_s,\vx_r;k) [\vP(\vx_s,\vx_r) +
 \cO((kd)^{-1})].
\end{equation}
The approximation \eqref{eq:pgp} follows from the last equation and
\eqref{eq:Pxry}.
To prove the expression of the condition number, we write a SVD of the
projectors as follows $\vP(\vx_s,\vx_r) = \vU_{rs} \vU_{rs}^*$,
$\vP(\vx_j,\vy_0) = \vU_{j0}\vU_{j0}^*$,
where the $3\times 2$ unitary matrices $\vU_{rs}$ and $\vU_{j0}$ are
\begin{equation}
 \vU_{rs} = \left[\vz,
 \frac{\vz\times(\vx_r-\vx_s)}{|\vz\times(\vx_r-\vx_s)|}\right],
 ~\text{and}~
 \vU_{j0} = \left[\vz,
 \frac{\vz\times(\vy_0-\vx_j)}{|\vz\times(\vy_0-\vx_j)|}\right],
 ~j=r,s,
\end{equation}
and $\vz$ is a unit length vector in $\{ \vy_0 - \vx_s, \vy_0 - \vx_r
\}^\perp$.  With this choice, a direct calculation gives
$\vU_{j0}^* \vU_{rs} = \diag(1,\cos(\theta_j))$, $j=r,s$. Therefore we
get  the SVD (up to a sign):
\begin{equation}
\vP(\vx_r,\vy_0) \vP(\vx_s,\vx_r) \vP(\vy_0,\vx_s) = \vU_{r0}
\diag(1,\cos(\theta_r)\cos(\theta_s)) \vU_{s0}^*.
\end{equation}
The rank and condition number identity follows.
\end{proof}

\section{Proof of Lemma \ref{lem.asymHr}}\label{sec:prooflem1} 

\begin{proof}
To prove \eqref{eq.decreascrossrange1}, it is more convenient to use the
rescaled array $\widetilde{\cA} = a^{-1} \cA$, where $
\operatorname{mes} \widetilde{\cA} =\cO(1)$.  We then follow the
standard method based on integration by parts to study the asymptotic
behavior of oscillating integrals (see e.g.\cite{Wong:2014:AAI}).

To this aim, we first rewrite the relation \eqref{eq.defhraproxs} as: 
\begin{equation}\label{eq.statphae}
\widetilde{\vH}_r(\vy,\vy';k)=\frac{a^2 \exp[\imath k(\eta'-\eta)]  }{(4\pi L)^2}    \, \vP_{\parallel} \, \int_{\widetilde{\cA}} d  \widetilde{\vx}_{r,\parallel} \exp[\imath f(\widetilde{\vx}_r)] \, 
\end{equation}
where $f$ is defined on the rescaled array $\widetilde{\cA}$ by $$ f(\widetilde{\vx}_{r,\parallel})=\frac{k a}{L}  \, \big[\widetilde{\vx}_r\cdot(\vy_{\parallel}-\vy'_{\parallel})\big] \mbox{ with }  \widetilde{\vx}_{r,\parallel}=\frac{\vx_{r,\parallel}}{a}\in \widetilde{\cA}.
$$
Using the identity (which holds since $\nabla_{ \widetilde{\vx}_{r,\parallel}} f=(k a /L) (\vy_{\parallel}-\vy'_{\parallel})$ is constant):
$$
\exp[\imath f]=\frac{1}{\imath} \operatorname{div}_{\widetilde{\vx}_r}\left( \exp[\imath f]\
\frac{\nabla{f}_{\widetilde{\vx}_r}}{|\nabla{f}_{\widetilde{\vx}_r}|^2}\right),
$$
and the divergence theorem applied to \eqref{eq.statphae} yields
$$
\widetilde{\vH}_r(\vy,\vy';k)=\frac{a^2 \exp[\imath k(\eta'-\eta)]  }{(4\pi L)^2}\, \vP_{\parallel}    \, \frac{L}{k a|\vy_{\parallel}-\vy'_{\parallel}|}  \, \int_{\partial \widetilde{\cA}}  \exp[\imath f]\
\frac{\nabla{f}_{\widetilde{\vx}_r}}{\imath
|\nabla{f}_{\widetilde{\vx}_r}|} \cdot \vn,
$$
where $\partial\widetilde{\cA}$ is the boundary of $\widetilde{\cA}$ and
$\vn$ is the outward pointing normal vector to $\partial\widetilde{\cA}$.
We conclude using the Cauchy-Schwarz inequality to get
$$
\|\widetilde{\vH}_r(\vy,\vy';k) \|\leq \frac{\operatorname{mes}\widetilde{\cA} \, a^2}{(4\pi L)^2}    \|\vP_{\parallel} \| \, \frac{L}{ a k |\vy_{\parallel}-\vy'_{\parallel}|}= \frac{a^2}{L^2}  \, \cO\Big(  \frac{ L}{
 a \, k   |\vy_{\parallel}-\vy'_{\parallel} |}\Big).
$$

Finally, the formula  \eqref{eq.asymptappoxHr} follows from an immediate  computation of the expression  of \eqref{eq.defhraproxs} when $\vy'=\vy$.
\end{proof}

\bibliographystyle{siamplain}
\bibliography{stokesbib}
\end{document}